\documentclass[11pt]{article}
\usepackage{amsmath, amssymb, amsfonts, amsthm, authblk, color, graphicx, enumitem, mathrsfs, indentfirst}
\usepackage{hyperref, cleveref}
\hypersetup{colorlinks=true, linkcolor=blue, filecolor=magenta, urlcolor=cyan}
\usepackage[textwidth=6in,textheight=9in]{geometry}
\usepackage{float}

\allowdisplaybreaks

\newtheorem{theorem}{Theorem}
\newtheorem{proposition}{Proposition}
\newtheorem{lemma}{Lemma}[section]

\newtheorem{remark}{Remark}[section]

\newtheorem{problem}{Problem}

\newtheorem{example}{Example}

\numberwithin{equation}{section}

\newcommand{\keywords}[1]{\small\textbf{\textit{Keywords---}}#1}

\title{Mathematical analysis and symmetric fractional-order reduction method for diffusion-wave equations}

\author[1]{Dakang Cen}
\author[2]{Caixia Ou\footnote{Corresponding author 1: oucaixiaywxk@163.com.}}
\author[3]{Seakweng Vong\footnote{Corresponding author 2: swvong@um.edu.mo, supported by University of Macau (File No. MYRG-GRG2025-00077-FST, MYRG-GRG2024-00100-FST-UMDF).}}
\affil[1]{Department of Mathematics, Southern University of Science and Technology, Shenzhen, China}
\affil[2,3]{Department of Mathematics, University of Macau, Macao, China.}
\date{}
\begin{document}
\maketitle

\abstract{In this work, our aim is to introduce a symmetric fractional-order reduction (SFOR) method to develop  numerical algorithms on nonuniform temporal meshes for fractional wave equations under lower regularity assumptions. The $L$-type  methods--including $L1$ and $L2$-$1_\sigma$ schemes--are specifically designed for diffusion-wave equations, and we propose novel optimal parameter selections tailored to nonuniform meshes. Finally, several numerical experiments are conducted to validate the efficiency and accuracy of the algorithms.}



\section{Introduction}
Assuming that $\alpha\in(1,2)$ and $\Omega:=[0,L] $, we consider the following diffusion-wave equation:
\begin{equation}\label{eq-gov}
\begin{cases}
  \partial_t^\alpha u - u_{xx} = g(x,t), & (x,t) \in \Omega \times (0,T), \\
  u(x,0) = a_0(x), & x \in \Omega, \\
  \frac{\partial}{\partial t}u(x,0) = a_1(x), & x \in \Omega, \\
  u(x,t) = 0, & (x,t) \in \partial \Omega \times (0,T),
\end{cases}
\end{equation}
where the operator $\partial_t^\delta$ is known as the Caputo derivative of order $\delta$:
\[\partial_t^\delta u(t):= (\mathcal{I}^{n-\delta}u^{(n)}
)(t)~\mbox{for}~ t>0 ~\mbox{and}~ n-1<\delta< n,\]
in which $\mathcal{I}^{\beta}$ represents the Riemann–Liouville fractional integral of order $\beta$:
\[\mathcal{I}^{\beta}u(t):= 
 \int_0^t \omega_\beta(t-s)u(s) ds~\mbox{with}~\omega_\beta(t)=\frac{t^{\beta-1}}{\Gamma(\beta)}, ~\beta>0.
\]

Particularly, the diffusion-wave equation, also referred to as the time-fractional wave equation, provides a unified framework for modeling evolutionary processes that interpolate between classical diffusion and wave propagation, such as the propagation of mechanical waves in viscoelastic media \cite{Mainardi,Mainardi1}. For the sake of solving the fractional diffusion-wave equations numerically, a great many of researchers have developed various high-order numerical methods. In 2006, Sun and Wu \cite{Sun1} utilized a standard order reduction method, i.e., by virtue of introducing an auxiliary function $v=u_t$, thereby recasting the original second-order-in-time equation into a first-order coupled system:
\begin{equation}\label{eq-or}
\begin{cases}
\partial_t
^{\alpha-1}v=u_{xx} + g(x, t),\\
v = u_t,
\end{cases}
\end{equation}
for $x\in\Omega, t\in(0, T]$. 
Then, the diffusion-wave equation can be solved following the standard framework of the $L1$ method on uniform grids. In 2014, Wang and Vong \cite{WangJCP} proposed a temporal second-order difference scheme based on weighted and shifted Gr\"{u}nwald difference operator (WSGD) for the following kind of fractional wave equation:  
\begin{align}\label{eq-inte}
u_t + \mathcal{I}^\beta u_{xx}(t) = g(t)~~\mbox{for}~ \beta \in(0, 1) ~\mbox{and}~t\in(0,T].    
\end{align} 
It can be observed that the above integro-differential problem is mathematically equivalent to the Eq. (\ref{eq-gov}) under proper assumptions on $g$ and initial data.
Afterwards, inspired by Alikhanov's work \cite{Alikhanov}, Sun $et.~al$ \cite{Sun} proposed some second-order difference schemes for both one- and two-dimensional time-fractional wave equations. However, the work mentioned above is validate under the assumption that the solution of the fractional model is sufficiently smooth.

Recently, a series of influential works have addressed the numerical approximation of Eq. (\ref{eq-gov}) in the presence of weakly singular solutions. Specially, there are two types of numerical framework for the problems: uniform \cite{Jin2016,Jin2016wave,Xu2007,Luo2019} and nonuniform \cite{Kopteva2019,LiaoH2018L1,Liao2021,Stynes2017} temporal grids. In 2016, Jin $et~al.$ \cite{Jin2016wave} conducted a rigorous analysis of convolution quadrature generated by backward difference formulas, establishing first- and second-order temporal convergence rates under suitable smoothness assumptions on the source term and initial data.    More recently, to address nonsmooth data scenarios, Luo $et~al.$ \cite{Luo2019} proposed a Petrov–Galerkin method achieving a temporal convergence rate of $(3-\alpha)/2$, while Li $et~al.$ \cite{LiB2020} developed a first-order time-stepping discontinuous Galerkin scheme--both methods specifically designed for problems with limited solution regularity. Notably, all aforementioned methods rely on uniform temporal grids. In contrast, Mustapha \& McLean \cite{Mustapha2014} and Mustapha \& Sch\"{o}tzau \cite{Mustapha2013} proposed discontinuous Galerkin time-stepping methods on nonuniform temporal meshes for the fractional wave equation (\ref{eq-inte}), achieving optimal convergence rates and robust temporal accuracy even for nonsmooth solutions.  Laplace transform methods and convolution quadrature methods
on uniform temporal steps were also discussed respectively. In 2021, the standard order reduction method (\ref{eq-or}) was extended  to the corresponding
nonuniform $L2$ scheme tailored for the linear diffusion-wave with weak singular solutions.
Nevertheless, a crucial property for the discrete coefficients--specifically, the positivity and monotonicity property stated in \cite[Assumption 4.1]{Shen2021}, which underpins the stability and convergence analysis--remains unverified. To circumvent this analytical limitation, Lyu and Vong proposed a novel order reduction technique, the symmetric fractional-order reduction method, to numerically solve diffusion-wave equations \cite{LyuP2022SFOR}. The main idea is presented in Lemma \ref{lem-sym}.
\begin{lemma}\label{lem-sym}
For $\alpha\in(1,2)$ and $u(t)\in C^1([0,T])\cap C^2((0,T])$, we have
\[
\partial_t^\alpha u(t)=\partial_t^{\frac{\alpha}{2}}(\partial_t^{\frac{\alpha}{2}}u(t))-u'(0)\omega_{2-\alpha}(t),
\]
where $\partial_t^\beta \psi(t) := \frac{1}{\Gamma(1 - \beta)} \int_0^t (t-\tau)^{-\beta} \psi'(\tau) d\tau$, $\beta\in(0,1)$.
\end{lemma}
Since $\lim_{_{t\rightarrow 0}}\partial_t^\beta \psi(t)\rightarrow0$ as $\psi(t)\in C^1[0,T]$. Let $\beta=\alpha/2$, then the model (\ref{eq-gov}) is transformed into the following form:
\begin{equation}\label{eq-gov-trans}
\begin{cases}
  \partial_t^\beta v - u_{xx} = a_1(x)\omega_{2-\alpha}(t)+g(x,t), & (x,t) \in \Omega \times (0,T), \\
  v  = \partial_t^\beta u, & (x,t) \in \Omega \times (0,T), \\
  u(x,0) =a_0(x),~~ v(x,0)=0, & x \in \Omega, \\
  u(x,t) = v(x,t)=0, & (x,t) \in \partial \Omega \times (0,T),
\end{cases}
\end{equation}
where singular source $g(x,t):=t^{1-\alpha}f(x)$. Noting that $a_1(x)\omega_{2-\alpha}(t)$ and $g(x,t)$ $\rightarrow$ $\infty$ as $t\rightarrow0^+$.
A novel form (\ref{eq-gov-trans-novel}) is proposed by $\partial_t^\beta \omega_{2-\alpha/2}(t)=\omega_{2-\alpha}(t)$:
\begin{equation}\label{eq-gov-trans-novel}
\begin{cases}
  \partial_t^\beta z - u_{xx} = 0, & (x,t) \in \Omega \times (0,T), \\
  z  = \partial_t^\beta u-[a_1(x)+\Gamma{(2-\alpha)}f(x)]\omega_{2-\alpha/2}(t), & (x,t) \in \Omega \times (0,T), \\
  u(x,0) =a_0(x),~~ z(x,0)=0, & x \in \Omega, \\
  u(x,t) = z(x,t)=0, & (x,t) \in \partial \Omega \times (0,T).
\end{cases}
\end{equation}

In the following parts of this paper, (\ref{eq-gov-trans}) and (\ref{eq-gov-trans-novel}) are our investigated models. It should be noted that our work is not just an application based on \cite{LyuP2022SFOR}. The authors introduced the auxiliary functions $\textbf{u}=u-ta_1(x)$ and $\textbf{v}  = \partial_t^\beta\textbf{u}$ to make the singular term $a_1(x)\omega_{2-\alpha}(t)$ disappear, see Remark 2.2 in \cite{LyuP2022SFOR}. Following this way, they considered solving the following equation numerically:
\begin{equation}\label{eq-gov-trans-1}
\begin{cases}
  \partial_t^\beta\textbf{v} - \textbf{u}_{xx} = t(a_1)_{xx}(x)+g(x,t), & (x,t) \in \Omega \times (0,T), \\
  \textbf{v}  = \partial_t^\beta \textbf{u}, & (x,t) \in \Omega \times (0,T), \\
  \textbf{u}(x,0) =a_0(x),~~ \textbf{v}(x,0)=0, & x \in \Omega, \\
  \textbf{u}(x,t) = \textbf{v}(x,t)=0, & (x,t) \in \partial \Omega \times (0,T).
\end{cases}
\end{equation}
There are several numerical methods for the model (\ref{eq-gov-trans-1}) (or similar models) \cite{LyuP2021SFOR,An2022,Kundaliya,cendcds,chenlisha,duruilian,liuyang2025}. Until now, only smooth initial functions $a_1(x)\in C^2(\Omega)$ and non-singular source $g(x,t)$ have been considered in numerical experiments. But we notice that the original system can be achieved symmetric order reduction only under the condition that $u(t)\in C^1([0,T])\cap C^2((0,T])$ in the Lemma \ref{lem-sym}. Therefore, it natural for us to raise the following question and we will solve this problem rigorously. 
\begin{problem}
What about the cases $a_1(x)\notin C^2(\Omega)$? From the insight of PDEs theory, (\ref{eq-gov-trans}) or (\ref{eq-gov-trans-novel}) is same to (\ref{eq-gov-trans-1}). However, from the perspective of numerical algorithms, it may be preferable to solve (\ref{eq-gov-trans}) or (\ref{eq-gov-trans-novel}) numerically when $a_1(x)$ has weaker regularity. 
\end{problem}

 The restriction of (\ref{eq-gov-trans-1}), $\Delta a_1(x)$ for $x\in\Omega$, must be well-defined in practical computing. While our numerical frameworks (\ref{eq-gov-trans}) and (\ref{eq-gov-trans-novel}) relax this requirement. Furthermore, we observe that the optimal convergence rate is reached, even in the presence of $a_1(x)\omega_{2-\alpha}(t)$, $a_1(x)\in L^2(\Omega)$. Our main conclusions are as follows. 
\begin{proposition}\label{main-pro1}
If $a_0(x)\in \mathcal{D}((-\Delta )^\frac32)$, $a_1(x)\in \mathcal{D}((-\Delta )^\frac12)$ and $t^{\alpha-1}g(x,t)\in \mathcal{D}((-\Delta )^\frac12)$ for $t\in(0,T]$, then
\begin{align}
    &\|\partial_t^l(-\Delta)^{\frac12} u(t)\|_{L^2(\Omega)}\leq C(1+t^{1-\alpha/4-l}),\\
    &\|\partial_t^l v(t)\|_{L^2(\Omega)}\leq Ct^{1-\alpha/2-l},~~\|\partial_t^l z(t)\|_{L^2(\Omega)}\leq Ct^{\alpha/2-l}.
\end{align}
For $\forall t\in(0,T]$, $L1$ scheme of (\ref{eq-gov-trans}) reaches consistent optimal $H^1(\Omega)$ norm convergence rate when the parameter of the graded mesh is $r=\frac{4-\alpha}{2-\alpha}$. And, $L1$ scheme of (\ref{eq-gov-trans-novel}) reaches consistent optimal $H^1(\Omega)$ norm convergence rate when $r=\max\{\frac{4-\alpha}{\alpha},2\}$.
\end{proposition}

\begin{proposition}\label{main-pro2}
If $a_0(x)\in \mathcal{D}((-\Delta )^2)$, $a_1(x)\in \mathcal{D}((-\Delta ))$ and $t^{\alpha-1}g(x,t)\in \mathcal{D}((-\Delta ))$ for $t\in(0,T]$, then
\begin{align}
    &\|\partial_t^l(-\Delta)u(t)\|_{L^2(\Omega)}\leq C(1+t^{1-\alpha/4-l}), \\
    &\|\partial_t^l(-\Delta)^{\frac12} v(t)\|_{L^2(\Omega)}\leq Ct^{1-\alpha/2-l},
    ~~\|\partial_t^l(-\Delta)^{\frac12} z(t)\|_{L^2(\Omega)}\leq Ct^{\alpha/2-l}.
\end{align}
For $\forall t\in(0,T]$, $L2$-$1_\sigma$ scheme of (\ref{eq-gov-trans}) reaches optimal $H^1(\Omega)$ norm convergence rate when the parameter of the graded mesh is $r=\frac{4}{2-\alpha}$. And, $L2$-$1_\sigma$ scheme of (\ref{eq-gov-trans-novel}) reaches consistent optimal $H^1(\Omega)$ norm convergence rate when $r=\max\{\frac{4}{\alpha},\frac{8}{4-\alpha}\}$.
\end{proposition}

The above propositions discuss the $H^1(\Omega)$ norm convergence of numerical schemes.
For the case of discontinue $a_1(x)$, similar results can be derived from the $L^2(\Omega)$ norm.
\begin{proposition}\label{main-pro3}
Under the assumptions $a_0(x)\in H^2(\Omega)\cap H_0^1(\Omega)$, $a_1(x)\in L^2(\Omega)$ and $t^{\alpha-1}g(x,t)\in L^2(\Omega)$ for $t\in(0,T]$.
For $\forall t\in(0,T]$, $L1$ scheme of (\ref{eq-gov-trans}) reaches consistent optimal $L^2(\Omega)$ norm convergence rate when the parameter of the graded mesh is $r=\frac{4-\alpha}{2-\alpha}$. And, $L1$ scheme of (\ref{eq-gov-trans-novel}) reaches consistent optimal $L^2(\Omega)$ norm convergence rate when $r=\max\{\frac{4-\alpha}{\alpha},2\}$.
\end{proposition}

The structure of this paper is as follows. The stability of three kinds of equivalent models is investigated in Section \ref{sec-stability}. Detailed regularity theory is presented in Section \ref{sec-reg}. In Section \ref{sec-na}, two types of numerical algorithms are constructed to solve the considered models. Numerical experiments are carried out to verify our theoretical results in Section \ref{sec-num}.

\section{Stability for (\ref{eq-gov-trans}), (\ref{eq-gov-trans-novel}) and (\ref{eq-gov-trans-1})}\label{sec-stability}
In this section, we prove that for different models, their stability is presented based on different regularity cases of $a_1(x)$. An interesting finding shows that smoother $a_1(x)$ matches the power function $t^p$ with the larger index $p$. 
Specifically, $p$ takes $1-\alpha$, $1-\alpha/2$ and $1$ for source terms in (\ref{eq-gov-trans}), (\ref{eq-gov-trans-novel}) and (\ref{eq-gov-trans-1}).

A useful property of fractional derivative is presented. The proof can be found in theorem 3.4 (ii) in \cite{Kubica}. Let $L^2(\Omega)$ be the square-integrable function space with inner product $(\cdot,\cdot)_{L^2(\Omega)}$ (or $(\cdot,\cdot)$ for short). 
\begin{lemma}\label{Coercivity}(Coercivity) For any function $h\in C^1[0,T]$ with $\alpha\in(0,1)$, one has the inequity 
    $$\frac{2}{\Gamma{(\alpha)}}\int_0^t(t-s)^{\alpha-1}h(s)\partial_t^\alpha h(s)ds\geq h^2(t)-h^2(0).$$
\end{lemma}

\begin{lemma}
For the stability of (\ref{eq-gov-trans}), one has    
$$\|v\| + \|\nabla u\|\leq C(\|\nabla a_0\|+\|a_1\|+\|f\|),$$
where $C$ is a constant depends on $t$ and $\alpha$.
\end{lemma}
\begin{proof}
Taking the inner product $(\cdot,\cdot)_{L^2(\Omega)}$ with $v$ and $\Delta u$ for the first two equations of (\ref{eq-gov-trans}), respectively. It gives that
\begin{align*}
(\partial_t^\beta v,v) - (\Delta u,v) &= \omega_{2-\alpha}(t)(a_1,v)+\Gamma{(2-\alpha)}\omega_{2-\alpha}(t)(f,v), \\
(v,\Delta u)  &= (\partial_t^\beta u,\Delta u). 
\end{align*}
Adding above equations, it comes that
\begin{align*}
(\partial_t^\beta v,v) + (\partial_t^\beta \nabla u,\nabla u)
&= \omega_{2-\alpha}(t)(a_1,v)+\Gamma{(2-\alpha)}\omega_{2-\alpha}(t)(f,v)\\
&\le \omega_{2-\alpha}(t)(\|a_1\|+\Gamma{(2-\alpha)}\|f\|)(\|v\|+\|\nabla u\|).
\end{align*}
By Lemma \ref{Coercivity}, one has
\begin{align*}
\|v\|^2 + \|\nabla u\|^2
\le& \|\nabla a_0\|^2+2(\|a_1\|+\Gamma{(2-\alpha)}\|f\|)\mathcal{I}^{\beta}\omega_{2-\alpha}(t)(\|v\|+\|\nabla u\|),
\end{align*}
where $\mathcal{I}^{\beta}\omega_{2-\alpha}(t)=\omega_{2-\alpha/2}(t)$.
The desired result follows by the triangle inequality.
\end{proof}

\begin{lemma}
For the stability of (\ref{eq-gov-trans-novel}), one has    
$$\|z\| + \|\nabla u\|\leq C(\|\nabla a_0\|+\|\nabla a_1\|+\|\nabla f\|),$$
where $C$ is a constant depends on $t$ and $\alpha$.
\end{lemma}
\begin{proof}
Acting $\nabla$ on the second equation of (\ref{eq-gov-trans-novel}). Taking the inner product $(\cdot,\cdot)_{L^2(\Omega)}$ with $z$ and $-\nabla u$ for the first two equations, respectively. It gives that
\begin{align*}
(\partial_t^\beta z,z) - (\Delta u,z) &= 0, \\
-(\nabla z,\nabla u)  &= -(\partial_t^\beta \nabla u,\nabla u)+\omega_{2-\alpha/2}(t)(\nabla a_1+\Gamma{(2-\alpha)}\nabla f,\nabla u). 
\end{align*}
Adding above equations, it comes that
\begin{align*}
(\partial_t^\beta z,z) + (\partial_t^\beta \nabla u,\nabla u)
&= \omega_{2-\alpha/2}(t)(\nabla a_1+\Gamma{(2-\alpha)}\nabla f,\nabla u)\\
&\le \omega_{2-\alpha/2}(t)(\|\nabla a_1\|+\Gamma{(2-\alpha)}\|\nabla f\|)(\|z\|+\|\nabla u\|).
\end{align*}
By Lemma \ref{Coercivity}, one has
\begin{align*}
\|z\|^2 + \|\nabla u\|^2
\le& \|\nabla a_0\|+2(\|\nabla a_1\|+\Gamma{(2-\alpha)}\|\nabla f\|)\mathcal{I}^{\beta}\omega_{2-\alpha/2}(t)(\|z\|+\|\nabla u\|),
\end{align*}
where $\mathcal{I}^{\beta}\omega_{2-\alpha/2}(t)=\omega_{2}(t)$.
The desired result follows by the triangle inequality.
\end{proof}

\begin{remark}
Similarly, one has the stability of (\ref{eq-gov-trans-1}). That is  $$\|\nabla u\|\le C(\|\nabla a_0\|+\|\Delta a_1\|+\|f\|+\|\nabla a_1\|).$$
\end{remark}

\section{Regularity theory for the model}\label{sec-reg}
In this section, we recall the well-posedness result of the initial-boundary value problem \eqref{eq-gov}. For this, we make several settings. Let $H^1(\Omega)$, $H^2(\Omega)$ etc. be the usual Sobolev spaces.

The set $\{\lambda_k, \varphi_k\}_{k=1}^{\infty}$ constitutes the Dirichlet eigensystem of the elliptic operator $-\Delta: H^2(\Omega) \cap H_0^1(\Omega) \to L^2(\Omega)$, specifically,
\begin{equation*}
\begin{cases}
-\Delta \varphi_k = \lambda_k \varphi_k & \text{in } \Omega, \\
\varphi_k =0 & \text{on } \partial \Omega,
\end{cases}
\end{equation*}
where $\lambda_k$ is the eigenvalue of the operator $-\Delta$ and satisfies $0 < \lambda_1 \leq \lambda_2 \leq \ldots , \lambda_k \rightarrow \infty$ as $k \rightarrow \infty$, and $\varphi_k$ is the eigenfunction corresponding to the value $\lambda_k$ and $\{\varphi_k\}_{k=1}^\infty$ forms an orthonormal basis in $L^2(\Omega)$. We have the asymptotic behavior of the eigenvalue $\lambda_k\sim k^{2/d}$ as $k\to\infty$. Then for $\gamma\in\mathbb R$, the fractional power $(-\Delta )^{\gamma}$ can be defined
\begin{equation}\label{frac-power-define}
  (-\Delta )^{\gamma}\psi:=\sum_{k=1}^{\infty}\lambda_k^{\gamma}(\psi, \varphi _k)\varphi _k,\quad \psi \in D((-\Delta)^\gamma),
\end{equation}
where
$$\mathcal{D}((-\Delta )^{\gamma}):=\left \{ \psi\in L^{2}(\Omega); \sum_{k=1}^{\infty}\left |\lambda_k^{\gamma}(\psi, \varphi _k ) \right |^2<\infty  \right \}.
$$
The space $D((-\Delta)^\gamma)$ is a Hilbert space equipped with the inner product
\begin{equation*}
( \psi, \phi)_{D((-\Delta)^\gamma)}  = \left ((-\Delta)^\gamma \psi, (-\Delta)^\gamma\phi \right )_{L^2(\Omega)}.
\end{equation*}
Moreover, we define the norm
$$
\begin{aligned}
\left \| \psi  \right \|_{\mathcal{D}((-\Delta )^\gamma) }
&= \left ((-\Delta)^\gamma \psi, (-\Delta)^\gamma\psi \right )_{L^2(\Omega)}^{\frac12}
= \left ( \sum_{n=1}^{\infty}\left |\lambda_n^{\gamma}(\psi, \varphi_n )  \right |^2 \right )^\frac{1}{2}.
\end{aligned}
$$
For short, we also denote the inner product $(\cdot,\cdot)_{\mathcal D((-\Delta)^\gamma)}$ and the norm $\|\cdot\|_{\mathcal D((-\Delta)^\gamma)}$ as $(\cdot,\cdot)_\gamma$ and $\|\cdot\|_\gamma$ if no conflict occurs. Furthermore, it satisfies $\mathcal{D}((-\Delta )^\gamma)\subset{H^{2\gamma}(\Omega)}$ for $\gamma>0$. In particular, we have $\mathcal{D}((-\Delta )^\frac{1}{2}) = H_0^1(\Omega)$, $\mathcal{D}((-\Delta )^{-\frac{1}{2}}) = H^{-1}(\Omega)$ and the norm equivalence $\left \| \cdot   \right \| _{\mathcal{D}((-\Delta )^\gamma  ) }\sim \left \| \cdot  \right \| _{H^{2\gamma}(\Omega)}$ with $\gamma=\pm \frac12$.

The regularity of the solution is based on the boundedness of the Mittag-Leffler functions in Lemma \ref{lem-ml-asymp}.
\begin{lemma}(\cite{podlubny1998fract})\label{lem-ml-asymp}
If $0<\alpha<2$, $\nu$ is an arbitrary complex number and $\mu$ is an arbitrary real number such that
$$
  \frac{\pi\alpha}2 <\mu<\min\{\pi,\pi\alpha\},
$$
then
\begin{align*}
|E_{\alpha,\nu}(z)| \le \frac{C}{1+|z|} ,\quad  \mu\le |\arg z|\le\pi,
\end{align*}
where $E_{\alpha,\nu}(z) = \sum_{k=0}^{\infty} \frac{z^k}{\Gamma(\alpha k + \nu)}, ~ z \in \mathbb{C}$.
\end{lemma}

\begin{lemma}\label{lem-reg-L2-2}
If $a_0\in D((-\Delta)^{\gamma_1})$, $\gamma_1\in(\frac{1}{\alpha},1]$ and $a_1,f\in D((-\Delta)^{\gamma_2})$, $\gamma_2\in(0,1]$, then
\[
\|\partial_{t}^mu\|_{L^2} \leq C\big(1+t^{\gamma_1\alpha-m}\big)\|(-\Delta)^{\gamma_1}a_0\|_{L^2}
+C\big(1+t^{\gamma_2\alpha+1-m}\big)\Big(\|(-\Delta)^{\gamma_2}a_1\|_{L^2}+\|(-\Delta)^{\gamma_2}f\|_{L^2}\Big),~m=0,1,2,3.
\]
\end{lemma}
\begin{proof}
The solution to the problem (\ref{eq-gov}) can be expressed as:
\begin{align}\nonumber
u(x,t)=&\sum_{n=1}^{\infty}E_{\alpha,1}(-\lambda_nt^\alpha)(a_0,\varphi_n)\varphi_n(x)\\\label{eq-u-wave-01}
&+\sum_{n=1}^{\infty} t E_{\alpha,2}(-\lambda_n t^\alpha)\Big((a_1, \varphi_n)+\Gamma(2-\alpha)(f, \varphi_n)\Big) \varphi_n(x).
\end{align}
From (\ref{eq-u-wave-01}), it indicates that
\begin{align*} 
&\|u\|_{L^2}^2\\
=&\bigg\|\sum_{n=1}^{\infty}E_{\alpha,1}(-\lambda_nt^\alpha)(a_0,\varphi_n)\varphi_n(x) +\sum_{n=1}^{\infty}t E_{\alpha,2}(-\lambda_n t^\alpha)\Big((a_1, \varphi_n)+\Gamma(2-\alpha)(f, \varphi_n)\Big)\varphi_n(x)\bigg\|_{L^2}^2\\ 
\leq& C\sum_{n=1}^{\infty}E_{\alpha,1}(-\lambda_nt^\alpha)^2|(a_0,\varphi_n)|^2 +C\sum_{n=1}^\infty t^2E_{\alpha,2}(-\lambda_n t^\alpha)^2\Big(|(a_1, \varphi_n)|^2+|(f, \varphi_n)|^2\Big)\\ 
\le& C\|a_0\|_{L^2}^2+Ct\Big(\|a_1\|_{L^2}^2+\|f\|_{L^2}^2\Big). 
\end{align*}

By differentiating $u(x,t)$ in regard to $t$, it holds that
\begin{align*}
&\|\partial_{t}u\|_{L^2}^2\\
=&\bigg\|-\sum_{n=1}^{\infty}\lambda_nt^{\alpha-1}E_{\alpha,\alpha}(-\lambda_nt^\alpha)(a_0,\varphi_n)\varphi_n(x)\\
&-\sum_{n=1}^{\infty} E_{\alpha,1}(-\lambda_n t^\alpha)\Big((a_1, \varphi_n)+\Gamma(2-\alpha)(f, \varphi_n)\Big)  \varphi_n(x)\bigg\|_{L^2}^2\\
\le&C\sum_{n=1}^{\infty}\lambda_n^2t^{2\alpha-2}E_{\alpha,\alpha}(-\lambda_nt^\alpha)^2|(a_0,\varphi_n)|^2
+C\sum_{n=1}^{\infty}E_{\alpha,1}(-\lambda_nt^\alpha)^2\Big(|(a_1, \varphi_n)|^2+|(f, \varphi_n)|^2\Big)\\
=&C\sum_{n=1}^{\infty}(\lambda_nt^\alpha)^{2-2\gamma_1}t^{2\gamma_1\alpha-2}E_{\alpha,\alpha-1}(-\lambda_nt^\alpha)^2\lambda_n^{2\gamma_1}|(a_0,\varphi_n)|^2\\
&+C\sum_{n=1}^{\infty}E_{\alpha,1}(-\lambda_n t^\alpha)^2\Big(|(a_1, \varphi_n)|^2+|(f, \varphi_n)|^2\Big)\\
\le&Ct^{2\gamma_1\alpha-2}\sup_n\frac{(\lambda_nt^\alpha)^{2-2\gamma_1}}{(1+\lambda_nt^\alpha)^2}\sum_{n=1}^{\infty}\lambda_n^{2\gamma_1}|(a_0,\varphi_n)|^2\\
&+C\sum_{n=1}^{\infty}\Big(|(a_1, \varphi_n)|^2+|(f, \varphi_n)|^2\Big)\\
\le&Ct^{2\gamma_1\alpha-2}\|(-\Delta)^{\gamma_1}a_0\|_{L^2}^2 +C\Big(\|a_1\|_{L^2}^2+\|f\|_{L^2}^2\Big),
\end{align*}
where the last inequality holds since $\sup_n\frac{(\lambda_nt^\alpha)^{2-2\gamma_1}}{(1+\lambda_nt^\alpha)^2}<C$.   
Then, we have
\[
\|\partial_{t}u\|_{L^2} \leq Ct^{\gamma_1\alpha-1}\|(-\Delta)^{\gamma_1}a_0\|_{L^2}
+C\Big(\|a_1\|_{L^2}+\|f\|_{L^2}\Big).
\]

Acting derivative with respect to $t$ on $\partial_t u(x,t)$, we arrive at
\begin{align*} 
\|\partial_{tt}u\|_{L^2}^2 =&\bigg\|-\sum_{n=1}^{\infty}\lambda_nt^{\alpha-2}E_{\alpha,\alpha-1}(-\lambda_nt^\alpha)(a_0,\varphi_n)\varphi_n(x)\\ &-\sum_{n=1}^{\infty}\lambda_nt^{\alpha-1} E_{\alpha,\alpha}(-\lambda_n t^\alpha)\Big((a_1, \varphi_n)+\Gamma(2-\alpha)(f, \varphi_n)\Big)  \varphi_n(x)\bigg\|_{L^2}^2\\ \le&C\sum_{n=1}^{\infty}\lambda_n^2t^{2\alpha-4}E_{\alpha,\alpha-1}(-\lambda_nt^\alpha)^2|(a_0,\varphi_n)|^2\\ &+C\sum_{n=1}^{\infty}\lambda_n^2t^{2\alpha-2}E_{\alpha,\alpha}(-\lambda_nt^\alpha)^2\Big(|(a_1, \varphi_n)|^2+|(f, \varphi_n)|^2\Big)\\
=&C\sum_{n=1}^{\infty}(\lambda_nt^\alpha)^{2-2\gamma_1}t^{2\gamma_1\alpha-4}E_{\alpha,\alpha-1}(-\lambda_nt^\alpha)^2\lambda_n^{2\gamma_1}|(a_0,\varphi_n)|^2\\ &+C\sum_{n=1}^{\infty}(\lambda_nt^\alpha)^{2-2\gamma_2}t^{2\gamma_2\alpha-2}E_{\alpha,\alpha}(-\lambda_n t^\alpha)^2\lambda_n^{2\gamma_2}\Big(|(a_1, \varphi_n)|^2+|(f, \varphi_n)|^2\Big)\\ \le&Ct^{2\gamma_1\alpha-4}\sup_n\frac{(\lambda_nt^\alpha)^{2-2\gamma_1}}{(1+\lambda_nt^\alpha)^2}\sum_{n=1}^{\infty}\lambda_n^{2\gamma_1}|(a_0,\varphi_n)|^2\\ &+Ct^{2\gamma_2\alpha-2}\sup_n\frac{(\lambda_nt^\alpha)^{2-2\gamma_2}}{(1+\lambda_nt^\alpha)^2}\sum_{n=1}^{\infty}\lambda_n^{2\gamma_2}\Big(|(a_1, \varphi_n)|^2+|(f, \varphi_n)|^2\Big)\\ \le&Ct^{2\gamma_1\alpha-4}\|(-\Delta)^{\gamma_1}a_0\|_{L^2}^2 +Ct^{2\gamma_2\alpha-2}\Big(\|(-\Delta)^{\gamma_2}a_1\|_{L^2}^2+\|(-\Delta)^{\gamma_2}f\|_{L^2}^2\Big). \end{align*} Then, we have \[ \|\partial_{tt}u\|_{L^2} \leq Ct^{\gamma_1\alpha-2}\|(-\Delta)^{\gamma_1}a_0\|_{L^2} +Ct^{\gamma_2\alpha-1}\Big(\|(-\Delta)^{\gamma_2}a_1\|_{L^2}+\|(-\Delta)^{\gamma_2}f\|_{L^2}\Big). \]

Similarly, we get
\begin{align*}
\|\partial_{ttt}u\|_{L^2}^2
=&\bigg\|-\sum_{n=1}^{\infty}\lambda_nt^{\alpha-3}E_{\alpha,\alpha-2}(-\lambda_nt^\alpha)(a_0,\varphi_n)\varphi_n(x)\\
&-\sum_{n=1}^{\infty} \lambda_nt^{\alpha-2}E_{\alpha,\alpha-1}(-\lambda_n t^\alpha)\Big((a_1, \varphi_n)+\Gamma(2-\alpha)(f, \varphi_n)\Big)  \varphi_n(x)\bigg\|_{L^2}^2\\
\le&C\sum_{n=1}^{\infty}(\lambda_nt^\alpha)^{2-2\gamma_1}t^{2\gamma_1\alpha-6}E_{\alpha,\alpha-2}(-\lambda_nt^\alpha)^2\lambda_n^{2\gamma_1}|(a_0,\varphi_n)|^2\\
&+C\sum_{n=1}^{\infty}(\lambda_nt^\alpha)^{2-2\gamma_2}t^{2\gamma_2\alpha-4}E_{\alpha,\alpha-1}(-\lambda_n t^\alpha)^2\lambda_n^{2\gamma_2}\Big(|(a_1, \varphi_n)|^2+|(f, \varphi_n)|^2\Big)\\
\le&Ct^{2\gamma_1\alpha-6}\sup_n(\lambda_nt^\alpha)^{2-2\gamma_1}E_{\alpha,\alpha-2}(-\lambda_n t^\alpha)^2\sum_{n=1}^{\infty}\lambda_n^{2\gamma_1}|(a_0,\varphi_n)|^2\\
&+Ct^{2\gamma_2\alpha-4}\sup_n(\lambda_nt^\alpha)^{2-2\gamma_2}E_{\alpha,\alpha-1}(-\lambda_n t^\alpha)^2\sum_{n=1}^{\infty}\lambda_n^{2\gamma_2}\Big(|(a_1, \varphi_n)|^2+|(f, \varphi_n)|^2\Big)\\
\le& Ct^{2\gamma_1\alpha-6}\|(-\Delta)^{\gamma_1}a_0\|_{L^2}^2 +Ct^{2\gamma_2\alpha-4}\Big(\|(-\Delta)^{\gamma_2}a_1\|_{L^2}^2+\|(-\Delta)^{\gamma_2}f\|_{L^2}^2\Big).
\end{align*}
Collecting all the above estimates, we finish the proof of the lemma.
\end{proof}

\begin{lemma}\label{thm-reg-delta-2}
If $a_0\in D((-\Delta)^{\gamma_1+\epsilon})$, $\gamma_1\in(\frac{1}{\alpha}+\frac14,\frac32]$ and $a_1,f\in D((-\Delta)^{\gamma_2+\epsilon})$, $\gamma_2\in(\frac{1}{4},1]$ and $0<\epsilon \ll 1$, then for $m=0,1,2,3,$
\[
\|\partial_t^m(-\Delta)^{\frac12}u\|_{L^2}\le C(1+t^{1-m-\frac{\alpha}{4}})\Big(\|(-\Delta)^{\gamma_1+\epsilon}a_0\|_{L^2}
+\|(-\Delta)^{\gamma_2+\epsilon}a_1\|_{L^2}+\|(-\Delta)^{\gamma_2+\epsilon}f\|_{L^2}\Big).
\]
\end{lemma}
\begin{proof}
For $\xi_2\in(0,1]$, it gives that
\begin{align*}
&\|(-\Delta)^{\frac12}u\|_{L^2}^2\\
=&\bigg\|\sum_{n=1}^{\infty}\lambda_n^{\frac12}E_{\alpha,1}(-\lambda_nt^\alpha)(a_0,\varphi_n)\varphi_n(x)
+\sum_{n=1}^{\infty}\lambda_n^{\frac12} t E_{\alpha,2}(-\lambda_n t^\alpha)\Big((a_1, \varphi_n)+\Gamma(2-\alpha)(f, \varphi_n)\Big)\varphi_n(x)\bigg\|_{L^2}^2\\
\leq& C\sum_{n=1}^{\infty}\lambda_nE_{\alpha,1}(-\lambda_nt^\alpha)^2|(a_0,\varphi_n)|^2
+C\sum_{n=1}^\infty\lambda_nt^2E_{\alpha,2}(-\lambda_n t^\alpha)^2\Big(|(a_1, \varphi_n)|^2+|(f, \varphi_n)|^2\Big)\\
=&C\sum_{n=1}^{\infty}\lambda_nE_{\alpha,1}(-\lambda_nt^\alpha)^2|(a_0,\varphi_n)|^2
+C\sum_{n=1}^\infty t^{\alpha(2\xi_2-2)+2}(\lambda_nt^\alpha)^{2-2\xi_2}E_{\alpha,2}(-\lambda_n t^\alpha)^2\lambda_n^{2\xi_2-1}\Big(|(a_1, \varphi_n)|^2+|(f, \varphi_n)|^2\Big)\\
\le&C\sum_{n=1}^{\infty}\lambda_nE_{\alpha,1}(-\lambda_nt^\alpha)^2|(a_0,\varphi_n)|^2
+C\sup_n\frac{(\lambda_nt^\alpha)^{2-2\xi_2}}{(1+\lambda_nt^\alpha)^2}t^{\alpha(2\xi_2-2)+2}\sum_{n=1}^\infty \lambda_n^{2\xi_2-1}\Big(|(a_1, \varphi_n)|^2+|(f, \varphi_n)|^2\Big)\\
\le& C\|(-\Delta)^{\frac12}a_0\|_{L^2}^2+Ct^{\alpha(2\xi_2-2)+2}
\Big(\|(-\Delta)^{\xi_2-\frac12}a_1\|_{L^2}^2+\|(-\Delta)^{\xi_2-\frac12}f\|_{L^2}^2\Big),
\end{align*}
where the last inequality holds since $\sup_n\frac{(\lambda_nt^\alpha)^{2-2\xi_2}}{(1+\lambda_nt^\alpha)^2}<C$.
We take  $\xi_2=\frac34\leq\gamma_2+\epsilon+\frac12$, it arrives that
\begin{align*}
\|(-\Delta)^{\frac12}u\|_{L^2}
&\le C\|(-\Delta)^{\frac12}a_0\|_{L^2}+ Ct^{1-\frac{\alpha}{4}}\Big(\|(-\Delta)^{\gamma_2+\epsilon}a_1\|_{L^2}+\|(-\Delta)^{\gamma_2+\epsilon}f\|_{L^2}\Big).
\end{align*}

Similarly, for $\xi_1,\xi_2\in(0,1]$, it gives that
\begin{align*}
&\|\partial_t(-\Delta)^{\frac12}u\|_{L^2}^2\\
=&\bigg\|\sum_{n=1}^{\infty}-\lambda_n^{\frac32}t^{\alpha-1}E_{\alpha,\alpha}(-\lambda_nt^\alpha)(a_0,\varphi_n)\varphi_n(x)
+\sum_{n=1}^{\infty}\lambda_n^{\frac12}E_{\alpha,1}(-\lambda_n t^\alpha)\Big((a_1, \varphi_n)+\Gamma(2-\alpha)(f, \varphi_n)\Big)\varphi_n(x)\bigg\|_{L^2}^2\\
\leq& C\sum_{n=1}^{\infty}\lambda_n^3t^{2\alpha-2}E_{\alpha,\alpha}(-\lambda_nt^\alpha)^2|(a_0,\varphi_n)|^2
+C\sum_{n=1}^\infty\lambda_nE_{\alpha,1}(-\lambda_n t^\alpha)^2\Big(|(a_1, \varphi_n)|^2+|(f, \varphi_n)|^2\Big)\\
=&C\sum_{n=1}^\infty t^{2\xi_1\alpha-2}(\lambda_nt^\alpha)^{2-2\xi_1}E_{\alpha,\alpha}(-\lambda_n t^\alpha)^2\lambda_n^{2\xi_1+1}|(a_0, \varphi_n)|^2\\
&+C\sum_{n=1}^\infty t^{\alpha(2\xi_2-2)}(\lambda_nt^\alpha)^{2-2\xi_2}E_{\alpha,1}(-\lambda_n t^\alpha)^2\lambda_n^{2\xi_2-1}\Big(|(a_1, \varphi_n)|^2+|(f, \varphi_n)|^2\Big)\\
\le&C\sup_n\frac{(\lambda_nt^\alpha)^{2-2\xi_1}}{(1+\lambda_nt^\alpha)^2}t^{2\xi_1\alpha-2}\sum_{n=1}^\infty \lambda_n^{2\xi_1+1}|(a_0, \varphi_n)|^2\\
&+C\sup_n\frac{(\lambda_nt^\alpha)^{2-2\xi_2}}{(1+\lambda_nt^\alpha)^2}t^{\alpha(2\xi_2-2)}\sum_{n=1}^\infty \lambda_n^{2\xi_2-1}\Big(|(a_1, \varphi_n)|^2+|(f, \varphi_n)|^2\Big)\\
\le& Ct^{2\xi_1\alpha-2}\|(-\Delta)^{\xi_1+\frac12}a_0\|_{L^2}^2+Ct^{\alpha(2\xi_2-2)}
\Big(\|(-\Delta)^{\xi_2-\frac12}a_1\|_{L^2}^2+\|(-\Delta)^{\xi_2-\frac12}f\|_{L^2}^2\Big).
\end{align*}
Taking $\xi_1=\frac{1}{\alpha}-\frac14\le\gamma_1+\epsilon-\frac12$ and $\xi_2=\frac34\leq\gamma_2+\epsilon+\frac12$, one has $\|\partial_t(-\Delta)^{\frac12}u\|_{L^2}\leq Ct^{-\frac{\alpha}{4}}\Big(\|(-\Delta)^{\gamma_1+\epsilon}a_0\|_{L^2}
+\|(-\Delta)^{\gamma_2+\epsilon}a_1\|_{L^2}+\|(-\Delta)^{\gamma_2+\epsilon}f\|_{L^2}\Big)$.

Some properties of $E_{\alpha,\nu}(-\lambda_nt^\alpha)$, $\alpha>0$, $\nu\in\mathbb{R}$, are needed to deduce the estimates of $\|\partial_t^m(-\Delta)^{\frac12}u\|_{L^2}$, $m=2,3$,
\begin{align*}
\partial_t[tE_{\alpha,1}(-\lambda_n t^\alpha)]
&= \partial_t\bigg[t\sum_{k=0}^\infty\frac{(-\lambda_nt^\alpha)^k}{\Gamma{(k\alpha+1)}}\bigg]\\
&= \partial_t\bigg[t+\sum_{k=1}^\infty\frac{(-\lambda_n)^kt^{k\alpha+1}}{\Gamma{(k\alpha+1)}}\bigg]\\
&= 1+\sum_{k=1}^\infty\frac{(-\lambda_nt^{\alpha})^k}{\Gamma{(k\alpha+1)}}+\sum_{k=1}^\infty\frac{(-\lambda_nt^{\alpha})^k}{\Gamma{(k\alpha)}}\\
&:=E_{\alpha,1}(-\lambda_nt^\alpha)+E_{\alpha,0}(-\lambda_nt^\alpha),
\end{align*}
and we get
\begin{align*}
\partial_t[E_{\alpha,1}(-\lambda_n t^\alpha)]
&=\partial_t[t^{-1}tE_{\alpha,1}(-\lambda_n t^\alpha)]\\
&=-t^{-1}E_{\alpha,1}(-\lambda_n t^\alpha)+t^{-1}\partial_t[tE_{\alpha,1}(-\lambda_n t^\alpha)]\\
&=-t^{-1}E_{\alpha,1}(-\lambda_n t^\alpha)+t^{-1}\big(E_{\alpha,1}(-\lambda_nt^\alpha)+E_{\alpha,0}(-\lambda_nt^\alpha)\big)\\
&=t^{-1}E_{\alpha,0}(-\lambda_nt^\alpha).
\end{align*}
Based on above results, it holds that
\begin{align*}
&\|\partial_t^2(-\Delta)^{\frac12}u\|_{L^2}^2\\
=&\bigg\|\sum_{n=1}^{\infty}-\lambda_n^{\frac32}t^{\alpha-2}E_{\alpha,\alpha-1}(-\lambda_nt^\alpha)(a_0,\varphi_n)\varphi_n(x)
+\sum_{n=1}^{\infty} \lambda_n^{\frac12}t^{-1}E_{\alpha,0}(-\lambda_n t^\alpha)\Big((a_1, \varphi_n)+\Gamma(2-\alpha)(f, \varphi_n)\Big)\varphi_n(x)\bigg\|_{L^2}^2\\
\leq& C\sum_{n=1}^{\infty}\lambda_n^3t^{2\alpha-4}E_{\alpha,\alpha-1}(-\lambda_nt^\alpha)^2|(a_0,\varphi_n)|^2
+C\sum_{n=1}^\infty\lambda_nt^{-2}E_{\alpha,0}(-\lambda_n t^\alpha)^2\Big(|(a_1, \varphi_n)|^2+|(f, \varphi_n)|^2\Big)\\
=&C\sum_{n=1}^\infty t^{\alpha(2\xi_2-2)-2}(\lambda_nt^\alpha)^{2-2\xi_2}E_{\alpha,0}(-\lambda_n t^\alpha)^2\lambda_n^{2\xi_2-1}\Big(|(a_1, \varphi_n)|^2+|(f, \varphi_n)|^2\Big)\\
\le&C\sup_n\frac{(\lambda_nt^\alpha)^{2-2\xi_1}}{(1+\lambda_nt^\alpha)^2}t^{2\xi_1-4}\sum_{n=1}^\infty \lambda_n^{2\xi_3+1}|(a_0, \varphi_n)|^2\\
&+C\sup_n\frac{(\lambda_nt^\alpha)^{2-2\xi_2}}{(1+\lambda_nt^\alpha)^2}t^{\alpha(2\xi_2-2)-2}\sum_{n=1}^\infty \lambda_n^{2\xi_2-1}\Big(|(a_1, \varphi_n)|^2+|(f, \varphi_n)|^2\Big)\\
\le& Ct^{2\xi_1\alpha-4}\|(-\Delta)^{\xi_1+\frac12}a_0\|_{L^2}^2+ Ct^{\alpha(2\xi_2-2)-2}\Big(\|(-\Delta)^{\xi_2-\frac12}a_1\|_{L^2}^2+\|(-\Delta)^{\xi_2-\frac12}f\|_{L^2}^2\Big),
\end{align*}
where $\xi_1,\xi_2\in(0,1]$, taking $\xi_1=\frac{1}{\alpha}-\frac14\le\gamma_1+\epsilon-\frac12$ and $\xi_2=\frac34\leq\gamma_2+\epsilon+\frac12$, i.e.
\begin{align*}
\|\partial_t^2(-\Delta)^{\frac12}u\|_{L^2}
&\le Ct^{-\frac{\alpha}{4}-1}\Big(\|(-\Delta)^{\gamma_1+\epsilon}a_0\|_{L^2}
+\|(-\Delta)^{\gamma_2+\epsilon}a_1\|_{L^2}+\|(-\Delta)^{\gamma_2+\epsilon}f\|_{L^2}\Big).
\end{align*}

In a similar fashion, we get
\begin{align*}
\|\partial_t^3(-\Delta)^{\frac12}u\|_{L^2}
&\le Ct^{-\frac{\alpha}{4}-2}\Big(\|(-\Delta)^{\gamma_1+\epsilon}a_0\|_{L^2}
+\|(-\Delta)^{\gamma_2+\epsilon}a_1\|_{L^2}+\|(-\Delta)^{\gamma_2+\epsilon}f\|_{L^2}\Big).
\end{align*}
The proof is completed.
\end{proof}

\begin{lemma}\label{thm-reg-delta} If $a_0\in D((-\Delta)^{\gamma_1+\epsilon+\frac12})$, $\gamma_1\in(\frac{1}{\alpha}+\frac14,\frac32]$ and $a_1,f\in D((-\Delta)^{\gamma_2+\epsilon+\frac12})$, $\gamma_2\in(\frac{1}{4},1]$ and $0<\epsilon \ll 1$, then for $m=0,1,2,3,$
\[
\|\partial_t^m(-\Delta)u\|_{L^2}\le C(1+t^{1-m-\frac{\alpha}{4}})\Big(\|(-\Delta)^{\gamma_1+\epsilon+\frac12}a_0\|_{L^2}
+\|(-\Delta)^{\gamma_2+\epsilon+\frac12}a_1\|_{L^2}+\|(-\Delta)^{\gamma_2+\epsilon+\frac12}f\|_{L^2}\Big).
\]
\end{lemma}
\begin{proof}
The desired results are directly obtained following the idea in Lemma \ref{thm-reg-delta-2}.    
\end{proof}

\begin{lemma}\label{thm-reg-L2-v}
If $a_0\in D((-\Delta)^{\gamma})$, $\gamma\in(\frac{1}{\alpha},1]$ and $a_1,f\in L^2(\Omega)$, then
\[
\|\partial_t^mv\|_{L^2}\le Ct^{\gamma\alpha-\beta-m} \|(-\Delta)^{\gamma}a_0\|_{L^2}+Ct^{1-\beta-m}\Big(\|a_1\|_{L^2}+\|f\|_{L^2}\Big),~~m=0,1,2,3.
\]
\end{lemma}
\begin{proof}
From (\ref{eq-u-wave-01}), indicates that
\begin{align*}
\|v\|_{L^2}^2
=&\bigg\|\frac{1}{\Gamma{(1-\beta)}}\int_0^t(t-s)^{-\beta}\partial_su(s)ds\bigg\|_{L^2}^2\\
=&\bigg\|\frac{1}{\Gamma{(1-\beta)}}\sum_{n=1}^{\infty}  \int_0^t(t-s)^{-\beta}\bigg(-\lambda_ns^{\alpha-1}E_{\alpha,\alpha}(-\lambda_ns^\alpha)(a_0,\varphi_n)\\
&+E_{\alpha,1}(-\lambda_n s^\alpha)\Big((a_1, \varphi_n)+\Gamma(2-\alpha)(f, \varphi_n)\Big)\bigg)ds\varphi_n(x)\bigg\|_{L^2}^2\\
=&\bigg\|-\sum_{n=1}^{\infty}\lambda_nt^{\alpha-\beta}E_{\alpha,\alpha+1-\beta}(-\lambda_n t^\alpha)(a_0, \varphi_n)\varphi_n(x)
+\sum_{n=1}^{\infty}t^{1-\beta}E_{\alpha,2-\beta}(-\lambda_n t^\alpha)\Big((a_1, \varphi_n)+(f, \varphi_n)\Big)\varphi_n(x)\bigg\|_{L^2}^2\\
\le&C\sum_{n=1}^{\infty}\lambda_n^2t^{2(\alpha-\beta)}E_{\alpha,\alpha+1-\beta}(-\lambda_n t^\alpha)^2|(a_0, \varphi_n)|^2
+C\sum_{n=1}^{\infty}t^{2-2\beta}E_{\alpha,2-\beta}(-\lambda_n t^\alpha)^2\Big(|(a_1, \varphi_n)|^2+|(f, \varphi_n)|^2\Big)\\
\le&Ct^{2\gamma\alpha-2\beta}\sup_n\frac{(\lambda_nt^\alpha)^{2-2\gamma}}{(1+\lambda_nt^\alpha)^2}\sum_{n=1}^\infty \lambda_n^{2\gamma}|(a_0, \varphi_n)|^2
+t^{2-2\beta}\sup_nE_{\alpha,2-\beta}(-\lambda_n t^\alpha)^2\sum_{n=1}^{\infty}\Big(|(a_1, \varphi_n)|^2+|(f, \varphi_n)|^2\Big)\\
\le&Ct^{2(\gamma\alpha-\beta)} \|(-\Delta)^{\gamma}a_0\|_{L^2}^2+Ct^{2(1-\beta)}\Big(\|a_1\|_{L^2}^2+\|f\|_{L^2}^2\Big).
\end{align*}

Similarly, we get
\begin{align*}
\|\partial_t^mv\|_{L^2}^2
=&\bigg\|\sum_{n=1}^{\infty}-\lambda_nt^{\alpha-\beta-m}E_{\alpha,\alpha+1-\beta-m}(-\lambda_n t^\alpha)(a_0, \varphi_n)\varphi_n(x)\\
&+\sum_{n=1}^{\infty}t^{1-\beta-m}E_{\alpha,2-\beta-m}(-\lambda_n t^\alpha)\Big((a_1, \varphi_n)+(f, \varphi_n)\Big)\varphi_n(x)\bigg\|_{L^2}^2\\
\le&Ct^{2(\alpha-\beta-m)}\sum_{n=1}^{\infty}\lambda_n^2E_{\alpha,\alpha+1-\beta-m}(-\lambda_n t^\alpha)^2|(a_0, \varphi_n)|^2\\
&+Ct^{2(1-\beta-m)}\sum_{n=1}^{\infty}\big(E_{\alpha,2-\beta-m}(-\lambda_n t^\alpha)\big)^2\Big(|(a_1, \varphi_n)|^2+|(f, \varphi_n)|^2\Big)\\
\le&Ct^{2(\gamma\alpha-\beta-m)} \|(-\Delta)^{\gamma}a_0\|_{L^2}^2+Ct^{2(1 -\beta-m)}\Big(\|a_1\|_{L^2}^2+\|f\|_{L^2}^2\Big).
\end{align*}
The proof of the theorem is completed.
\end{proof}

\begin{lemma}\label{thm-reg-delta-v}
If $a_0\in D((-\Delta)^{\gamma})$, $\gamma\in(\frac{1}{\alpha}+\frac12,\frac32]$ and $a_1,f\in H^1(\Omega)$, then
\[
\|\partial_t^m(-\Delta)^{\frac12}v\|_{L^2}\le Ct^{1-\beta-m}\Big(\|(-\Delta)^{\gamma}a_0\|_{L^2}+ \|(-\Delta)^{\frac12}a_1\|_{L^2}+\|(-\Delta)^{\frac12}f\|_{L^2}\Big),~~m=0,1,2,3.
\]
\end{lemma}
\begin{proof}
From $\xi\in(0,1]$, indicates that
\begin{align*}
&\|(-\Delta)^{\frac12}v\|_{L^2}^2\\
=&\bigg\|-\sum_{n=1}^{\infty}\lambda_n^{\frac32}t^{\alpha-\beta}E_{\alpha,\alpha+1-\beta}(-\lambda_n t^\alpha)(a_0, \varphi_n)\varphi_n(x)
+\sum_{n=1}^{\infty}\lambda_n^{\frac12}t^{1-\beta}E_{\alpha,2-\beta}(-\lambda_n t^\alpha)\Big((a_1, \varphi_n)+(f, \varphi_n)\Big)\varphi_n(x)\bigg\|_{L^2}^2\\
\le&C\sum_{n=1}^{\infty}\lambda_n^3t^{2(\alpha-\beta)}E_{\alpha,\alpha+1-\beta}(-\lambda_n t^\alpha)^2|(a_0, \varphi_n)|^2
+C\sum_{n=1}^{\infty}\lambda_nt^{2-2\beta}E_{\alpha,2-\beta}(-\lambda_n t^\alpha)^2\Big(|(a_1, \varphi_n)|^2+|(f, \varphi_n)|^2\Big)\\
\le&C\sup_n\frac{(\lambda_nt^\alpha)^{2-2\xi}}{(1+\lambda_nt^\alpha)^2}t^{2\xi\alpha-2\beta}\sum_{n=1}^\infty \lambda_n^{2\xi+1}|(a_0, \varphi_n)|^2
+t^{2-2\beta}\sup_nE_{\alpha,2-\beta}(-\lambda_n t^\alpha)^2\sum_{n=1}^{\infty}\lambda_n\Big(|(a_1, \varphi_n)|^2+|(f, \varphi_n)|^2\Big)\\
\le&Ct^{2-2\beta}\Big(\|(-\Delta)^{\gamma}a_0\|_{L^2}^2+\|(-\Delta)^{\frac12}a_1\|_{L^2}^2+\|(-\Delta)^{\frac12}f\|_{L^2}^2\Big),
\end{align*}
 where the last inequality follows from the fact that $\xi=\frac{1}{\alpha}\le\gamma-\frac12$. Then we yield $\|(-\Delta)^{\frac12}v\|_{L^2}\leq Ct^{1-\beta}\Big(\|(-\Delta)^{\gamma}a_0\|_{L^2}+\|(-\Delta)^{\frac12}a_1\|_{L^2}+\|(-\Delta)^{\frac12}f\|_{L^2}\Big)$.

Similarly, we get
\begin{align*}
\|\partial_t^m(-\Delta)^{\frac12}v\|_{L^2}^2
=&\bigg\|-\sum_{n=1}^{\infty}\lambda_n^{\frac32}t^{\alpha-\beta-m}E_{\alpha,\alpha+1-\beta-m}(-\lambda_n t^\alpha)(a_0, \varphi_n)\varphi_n(x)\\
&+\sum_{n=1}^{\infty}\lambda_n^{\frac12}t^{1-\beta-m}E_{\alpha,2-\beta-m}(-\lambda_n t^\alpha)\Big((a_1, \varphi_n)+(f, \varphi_n)\Big)\varphi_n(x)\bigg\|_{L^2}^2\\
\le&Ct^{2(\alpha-\beta-m)}\sum_{n=1}^{\infty}\lambda_n^3E_{\alpha,\alpha+1-\beta-m}(-\lambda_n t^\alpha)^2|(a_0, \varphi_n)|^2\\
&+Ct^{2(1-\beta-m)}\sum_{n=1}^{\infty}\lambda_n\big(E_{\alpha,2-\beta-m}(-\lambda_n t^\alpha)\big)^2\Big(|(a_1, \varphi_n)|^2+|(f, \varphi_n)|^2\Big)\\
\le&Ct^{2(1-\beta-m)}\Big(\|(-\Delta)^{\gamma}a_0\|_{L^2}^2+\|(-\Delta)^{\frac12}a_1\|_{L^2}^2+\|(-\Delta)^{\frac12}f\|_{L^2}^2\Big).
\end{align*}
The proof of the theorem is completed.
\end{proof}

\begin{lemma}\label{thm-reg-L2-z}
If $a_0\in H^2(\Omega)$ and $a_1,f\in H^1(\Omega)$, then
\[ \|\partial_t^mz\|_{L^2}\le Ct^{\beta-m} \|(-\Delta)a_0\|_{L^2}+Ct^{1-m}\Big(\|(-\Delta)^{\frac12}a_1\|_{L^2}+\|(-\Delta)^{\frac12}f\|_{L^2}\Big),~~m=0,1,2,3. \]
\end{lemma}
\begin{proof}
The solution $z$ of system (\ref{eq-gov-trans-novel}) can be expressed as:
\begin{align}\nonumber
z(x,t) =& -\sum_{n=1}^{\infty} \lambda_n t^{\beta}  E_{\alpha, 1+\beta}(-\lambda_n t^\alpha)(a_0, \varphi_n)\varphi_n(x) \\\label{solution-z}
&-\sum_{n=1}^{\infty}\lambda_nt^{1+\beta} E_{\alpha, 2+\beta}(-\lambda_n t^\alpha)\big( (a_1, \varphi_n) + \Gamma(2-\alpha) (f, \varphi_n) \big)\varphi_n(x).    
\end{align}
From (\ref{eq-u-wave-01}), indicates that 
\begin{align*}
\|z\|_{L^2}^2
=&\bigg\|-\sum_{n=1}^{\infty}\lambda_n t^{\beta}E_{\alpha,1+\beta}(-\lambda_n t^\alpha)(a_0, \varphi_n)\varphi_n(x)\\ 
&-\sum_{n=1}^{\infty}\lambda_nt^{1+\beta}E_{\alpha,2+\beta}(-\lambda_n t^\alpha)\Big((a_1, \varphi_n)+(f, \varphi_n)\Big)\varphi_n(x)\bigg\|_{L^2}^2\\ 
\le&C\sum_{n=1}^{\infty}\lambda_n^2t^{2\beta}E_{\alpha,1+\beta}(-\lambda_n t^\alpha)^2|(a_0, \varphi_n)|^2\\ &+C\sum_{n=1}^{\infty}\lambda_n^2t^{2+2\beta}E_{\alpha,2+\beta}(-\lambda_n t^\alpha)^2\Big(|(a_1, \varphi_n)|^2+|(f,\varphi_n)|^2\Big)\\
\le&Ct^{2\beta}\sup_nE_{\alpha,1+\beta}(-\lambda_n t^\alpha)^2{(1+\lambda_nt^\alpha)^2}\sum_{n=1}^\infty \lambda_n^{2}|(a_0, \varphi_n)|^2\\
&+t^{2}\sup_n\frac{\lambda_nt^{\alpha}}{(1+\lambda_nt^\alpha)^2}\sum_{n=1}^{\infty}\Big(\lambda_n|(a_1, \varphi_n)|^2+\lambda_n|(f, \varphi_n)|^2\Big)\\
\le&Ct^{2\beta} \|(-\Delta)a_0\|_{L^2}^2+Ct^{2}\Big(\|(-\Delta)^{\frac12}a_1\|_{L^2}^2+\|(-\Delta)^{\frac12}f\|_{L^2}^2\Big). 
\end{align*} 

Similarly, we get \begin{align*} \|\partial_t^mz\|_{L^2}^2 
=&\bigg\|-\sum_{n=1}^{\infty}\lambda_nt^{\beta-m}E_{\alpha,1+\beta-m}(-\lambda_n t^\alpha)(a_0, \varphi_n)\varphi_n(x)\\
&-\sum_{n=1}^{\infty}\lambda_nt^{1+\beta-m}E_{\alpha,2+\beta-m}(-\lambda_n t^\alpha)\Big((a_1, \varphi_n)+(f, \varphi_n)\Big)\varphi_n(x)\bigg\|_{L^2}^2\\ 
\le&Ct^{2(\beta-m)}\sum_{n=1}^{\infty}\lambda_n^2E_{\alpha,1+\beta-m}(-\lambda_n t^\alpha)^2|(a_0, \varphi_n)|^2\\ 
&+Ct^{2(1+\beta-m)}\sum_{n=1}^{\infty}\lambda_n^2\big(E_{\alpha,2+\beta-m}(-\lambda_n t^\alpha)\big)^2\Big(|(a_1, \varphi_n)|^2+|(f, \varphi_n)|^2\Big)\\ \le&Ct^{2(\beta-m)} \|(-\Delta)a_0\|_{L^2}^2+Ct^{2(1 -m)}\Big(\|(-\Delta)^{\frac12}a_1\|_{L^2}^2+\|(-\Delta)^{\frac12}f\|_{L^2}^2\Big). 
\end{align*}
The proof of the theorem is completed.
\end{proof}

\begin{lemma}\label{thm-reg-delta-z} If $a_0\in H^3(\Omega)$ and $a_1,f\in H^2(\Omega)$, then 
\[ \|\partial_t^m(-\Delta)^{\frac12}z\|_{L^2}\le Ct^{\beta-m}\Big(\|(-\Delta)^{\frac32}a_0\|_{L^2}+ \|(-\Delta)a_1\|_{L^2}+\|(-\Delta)f\|_{L^2}\Big),~~m=0,1,2,3. \] 
\end{lemma}
\begin{proof}
The desired results are directly obtained following the idea in Lemma \ref{thm-reg-L2-z}.

\end{proof}

\section{Numerical Algorithms}\label{sec-na}
\subsection{Preliminary}
Our main concern is the time approximation of (\ref{eq-gov}). For a positive integer $N$, the interval $[0,T]$ is divided into $0=t_0<t_1<\cdots< t_{k-1}<t_k<\cdots<t_N=T$ with $$t_k=T(k/N)^r,~r\geq1,~0\leq k\leq N,$$ with $\tau_k=t_k-t_{k-1}$ being the time step size. We use $\rho_k:=\tau_k/\tau_{k+1} ~\mbox{for}~1\leq k\leq N-1$ and $\rho=\max_k\{\rho_k\}$ to denote the local time step-size ratios and the maximum ratio, respectively. Here and hereafter, $g^k$ denotes $g(t_k)$. Define the off-set time
points and grid functions
$t_{n-\theta}:= \theta t_{n-1}+(1-\theta)t_n$ and $g^{n-\theta}:= \theta g^{n-1}+(1-\theta)g^n,~1 \le n \le N$.
The Caputo derivative $\partial_t^\beta v(t_{n-\theta})$ can be formally approximated by the following discrete form with convolution structure:  
\begin{align}\label{b-1}
(\bar\partial_t^\beta v)^{n-\theta}:=\sum\limits_{k=1}^{n}A_{n-k}^{(n)}\nabla_{\tau}v^k,~\mbox{where }~\nabla_{\tau}v^k=v^k-v^{k-1}.    
\end{align}
The general discretization (\ref{b-1}) includes two practical ones. It leads to the $L1$ formula while
$\theta= 0$ (see also (\ref{L1})) and the Alikhanov formula while $\theta=\beta/2$ (see also (\ref{L2-1})). To efficiently solve the fractional diffusion-wave equation with lower weak singular or more complicated solutions, we next give more explicit formulations of these two classical approximations on nonuniform meshes, which have also been rigorously studied in \cite{LiaoH2018L1,LiaoH2021L2_1}.

\textbf{Nonuniform $L1$ formula} The $L1$ formula (\ref{L1}), one of the most classical discrete methods, is used to approximate the Caputo derivative on nonuniform meshes $\{t_n|0=t_0<t_1<\dots <t_N=T\}$. Denote $\tau_n=t_n-t_{n-1}$, $n\geq1$.
\begin{align}\nonumber
\bar\partial_t^\beta v(t_n)
:&= \sum_{k=1}^{n}\int_{t_{k-1}}^{t_k}\omega_{1-\beta}(t_n-s)\frac{v(t_k)-v(t_{k-1})}{\tau_k}ds\\ \label{L1}
&=\sum_{k=1}^nA_{n-k}^{(n)}\nabla_\tau v(t_k),
\end{align}
where $\nabla_\tau v(t_k)=v(t_k)-v(t_{k-1})$, $A_{n-k}^{(n)}:=\int_{t_{k-1}}^{t_k}\frac{\omega_{1-\beta}(t_n-s)}{\tau_k}ds$. 

\textbf{Nonuniform Alikhanov formula} 
Denote $\theta=\beta/2=\alpha/4$, and the discrete coefficients $a_{n-k}^{(n)}$ and $b_{n-k}^{(n)}$ are defined by
\begin{align*}
&a_{0}^{(n)}=\frac{1}{\tau_n}\int_{t_{n-1}}^{t_{n-\theta}}
\omega_{1-\beta}(t_{n-\theta}-s)ds;\\
&a_{n-k}^{(n)}=\frac{1}{\tau_k}\int_{t_{k-1}}^{t_k}
\omega_{1-\beta}(t_{n-\theta}-s)ds,\quad 1\leq k\leq n-1;\\
&b_{n-k}^{(n)}=\frac{2}{\tau_k(\tau_k+\tau_{k+1})}
\int_{t_{k-1}}^{t_k}(s-t_{k-1/2})
\omega_{1-\beta}(t_{n-\theta}-s)ds,\quad 1\leq k\leq n-1.
\end{align*}
The approximation to
the Caputo derivative $\partial_t^\beta v(t_{n-\theta})$ is defined by
\begin{align}\nonumber
(\bar\partial_t^\beta v)^{n-\theta}
&:=\int_{t_{n-1}}^{t_{n-\theta}}\omega_{1-\beta}(t_{n-\theta}-s)(\Pi_{1}v)'(s)ds
+\sum\limits_{k=1}^{n-1}\int_{t_{k-1}}^{t_{k}}\omega_{1-\beta}(t_{n-\theta}-s)
(\Pi_{2}v)'(s)ds\\\nonumber
&=a_0^{(n)}\nabla_\tau v^n+\sum\limits_{k=1}^{n-1}(a_{n-k}^{(n)}\nabla_\tau v^k+\rho_kb_{n-k}^{(n)}\nabla_\tau v^{k+1}-b_{n-k}^{(n)}\nabla_\tau v^k),\\\label{L2-1}
&=\sum\limits_{k=1}^{n}A_{n-k}^{(n)}\nabla_{\tau}v^k,
\end{align}
where $\Pi_{1}$ and $\Pi_{2}$ denote the linear interpolation operator and the quadratic interpolation operator, respectively, and the discrete convolution kernels $A_{n-k}^{(n)}$ are defined as follows: $A_0^{(1)}=a_0^{(1)}$ if $n=1$ and, for $n\geq2$,
\begin{align}\label{b-4}
\begin{array}{c}
A_{n-k}^{(n)}=\left\{
\begin{array}{ll}
a_0^{(n)}+\rho_{n-1}b_1^{(n)},&\mbox{for}~k=n,\\
a_{n-k}^{(n)}+\rho_{k-1}b_{n-k+1}^{(n)}-b_{n-k}^{(n)},
&\mbox{for}~ 2\leq k\leq n-1,\\
a_{n-1}^{(n)}-b_{n-1}^{(n)},&\mbox{for}~ k=1.
\end{array}\right.
\end{array}
\end{align} 

  In the rest of this paper, we will use the general form (\ref{b-1}) to represent the nonuniform $L1$ formula and $L2$-$1_\sigma$ formula while \( \theta = 0 \) and \( \theta = \beta/2 \), respectively. The following two related properties have been verified in \cite{LiaoH2018L1,LiaoH2021L2_1} for the discrete coefficients  \( A^{(n)}_{n-k} \)  of the nonuniform $L1$ formula (with $\pi_A = 1$) and the nonuniform Alikhanov formula
(with $\pi_A=11/4$ and $\rho=7/4$), which are required in the numerical analysis of corresponding algorithms:
\begin{itemize}
    \item[\textbf{A1.}] The discrete kernels are positive and monotone: $A_0^{(n)}\geq A_1^{(n)}\geq\cdots\geq A_{n-1}^{(n)}> 0;$
    \item[\textbf{A2.}]There is a constant $\pi_A$ such that $A_{n-k}^{(n)}
\geq\frac{1}{\pi_A\tau_k}\int_{t_{k-1}}^{t_k}\omega_{1-\beta}(t_n-s)\frac{{\rm d}s}{s},~1\leq k\leq n.$
\end{itemize}

Next, the time semi-discrete scheme will be given.

\subsection{Time semi-discrete scheme}
The corresponding numerical schemes of problems (\ref{eq-gov-trans}) and (\ref{eq-gov-trans-novel}) are as follows:
\begin{equation}\label{num-scheme-1}
\begin{cases}
  (\bar\partial_t^\beta V)^{n-\theta} - \Delta U^{n-\theta} = a_1(x)\omega_{2-\alpha}(t_{n-\theta})+t_{n-\theta}^{1-\alpha}f(x), & 1\le n \le N, \\
  V^{n-\theta}  = (\bar\partial_t^\beta U)^{n-\theta}, & 1\le n \le N, \\
  U(x,0) = V(x,0)=0, & x \in \Omega, \\
  U(x,t) = V(x,t)=0, & (x,t) \in \partial \Omega \times (0,T),
\end{cases}
\end{equation}
and
\begin{equation}\label{num-scheme-2} \begin{cases}   (\bar\partial_t^\beta Z)^{n-\theta} - \Delta U^{n-\theta} = 0, & 1\le n \le N, \\  
Z^{n-\theta}  = (\bar\partial_t^\beta U)^{n-\theta}-[a_1(x)+\Gamma{(2-\alpha)}f(x)]\omega_{2-\beta}(t_{n-\theta}), & 1\le n \le N, \\   U(x,0) = V(x,0)=0, & x \in \Omega,\\ U(x,t) = V(x,t)=0, & (x,t) \in \partial \Omega \times (0,T), \end{cases} \end{equation}
where $U^n$, $V^n$  and $Z^n$ are numerical solutions corresponding to $u(t_n)$, $v(t_n)$ and $z(t_n)$  in (\ref{eq-gov-trans}) and (\ref{eq-gov-trans-novel}).
Some important lemmas are introduced as following.
\begin{lemma}\label{FD-ineq}
(\cite{LiaoH2018L1})For $V^n$, $1\leq n\leq N$, one has
\begin{align*}
((\bar\partial_t^\beta V)^{n-\theta},V^{n-\theta})\geq\frac{1}{2}\sum\limits_{k=1}^{n}A_{n-k}^{(n)}\nabla_{\tau}
(\|V^k\|^2).
\end{align*}
\end{lemma}

\begin{lemma}\label{grownwall}
(\cite{LyuP2022SFOR})Let $(g^n)_{n=1}^N$ and $(\lambda_l)_{l=0}^{N-1}$ be given nonnegative sequences. Assume that there exists a constant $\Lambda$ such that $\Lambda\geq\sum_{l=0}^{N-1}\lambda_l$, and that the maximum step satisfies
$$\max_{1\leq n \leq N}\tau_n\leq \frac{1}{\sqrt[\beta]{4\pi_A\Gamma{(2-\beta)}\Lambda}}.$$
Then, for any nonnegative sequence $(v^k)_{k=0}^N$ and $(w^k)_{k=0}^N$ satisfying
$$\sum_{k=1}^nA_{n-k}^{(n)}\nabla_\tau\big[(v^k)^2+(w^k)^2\big]\leq\sum_{k=1}^n\lambda_{n-k}\big(v^{k-\theta}+w^{k-\theta}\big)^2+(v^{n-\theta}+w^{n-\theta})g^n,~~1\leq n\leq N,$$
it holds that
$$v^n+w^n\leq 4E_\beta(4\pi_A\Lambda t_n^\beta)\bigg(v^0+w^0+\max_{1\leq k\leq n}\sum_{j=1}^kP_{k-j}^{(k)}g^j\bigg),~~1\leq n\leq N,$$
where $P_0^{(n)}:=\frac{1}{A_0^{(n)}}$, $P_{n-j}^{(n)}:=\frac{1}{A_0^{(j)}}\sum_{k=j+1}^n(A_{k-j-1}^{(k)}-A_{k-j}^{(k)})P_{n-k}^{(n)}$, $1\leq j\leq n-1$, $E_\beta(z)=\sum_{k=0}^\infty\frac{z^k}{\Gamma{(1+k\beta)}}$ is the Mittag-Leffler function.
\end{lemma}

\begin{lemma}\label{P}
For the sequence $(P_{n-j}^{(n)})_{j=1}^n$, some properties are given in \cite{LiaoH2018L1}.
\begin{align*}
&\sum_{j=k}^nP_{n-j}^{(n)}A_{j-k}^{(j)}\equiv1,~~1\leq k\leq n,\\
&0\leq P_{n-j}^{(n)}\leq \pi_A\Gamma{(2-\beta)}\tau_j^\beta,~~\sum_{j=1}^nP_{n-j}^{(n)}\omega_{1-\beta}(t_j)\leq C,~~ 1\leq j\leq n\leq N.
\end{align*}
\end{lemma}

\subsection{Convergence analysis}
In this part, we propose convergence estimates for the numerical schemes (\ref{num-scheme-1}) and (\ref{num-scheme-2}). First, the convergence analysis of (\ref{num-scheme-1}) is considered under the following conditions: $a_0\in D((-\Delta)^{\gamma_1+\epsilon})$, $\gamma_1\in(\frac{1}{\alpha}+\frac{1}{4},\frac32]$ and $a_1\in D((-\Delta)^{\gamma_2+\epsilon})$, $\gamma_2\in(\frac{1}{4},1]$, $0<\epsilon \ll 1$. This is consistent with the feasibility conditions for the SFOR framework. 
Denote $\mathcal{R}_v^{n-\theta}:=(\bar\partial_t^\beta v)^{n-\theta}-(\partial_t^\beta v)^{n-\theta}$, $\mathcal{R}_u^{n-\theta}:=(\partial_t^\beta u)^{n-\theta}-(\bar\partial_t^\beta u)^{n-\theta}$, $R_v^{n-\theta}=v(t_{n-\theta})-\big(\theta v(t_{n-1})+(1-\theta)v(t_n)\big)$, $\Delta R_u^{n-\theta}=\Delta u(t_{n-\theta})-\big(\theta\Delta u(t_{n-1})+(1-\theta)\Delta u(t_n)\big)$.

\begin{equation}\nonumber
\begin{cases}
  (\bar\partial_t^\beta v)^{n-\theta} - \Delta u^{n-\theta} = a_1(x)\omega_{2-\alpha}(t_{n-\theta})+g(x,t_{n-\theta})+\mathcal{R}_v^{n-\theta}+\Delta R_u^{n-\theta}, & 1\le n\le N, \\
  v^{n-\theta}  = (\bar\partial_t^\beta u)^{n-\theta} + \mathcal{R}_u^{n-\theta}+R_v^{n-\theta}, & 1\le n\le N, \\
  u(x,0) = v(x,0)=0, & x \in \Omega, \\
  u(x,t) = v(x,t)=0, & (x,t) \in \partial \Omega \times (0,T).
\end{cases}
\end{equation}
Denote $\bar u^n:=u^n-U^n$ and $\bar v^n:=v^n-V^n$. One has the error system (\ref{error-system}):
\begin{equation}\label{error-system}
\begin{cases}
  (\bar\partial_t^\beta \bar v)^{n-\theta} - \Delta \bar u^{n-\theta} = \mathcal{R}_v^{n-\theta}+\Delta R_u^{n-\theta}, & 1\le n\le N, \\
  \bar v^{n-\theta}  = (\bar\partial_t^\beta \bar u)^{n-\theta} + \mathcal{R}_u^{n-\theta}+R_v^{n-\theta}, & 1\le n\le N, \\
  \bar u(x,0) = \bar v(x,0)=0, & x \in \Omega, \\
  \bar u(x,t) = \bar v(x,t)=0, & (x,t) \in \partial \Omega \times (0,T).
\end{cases}
\end{equation}

\textbf{For $L1$ approximation $(\theta=0)$:} The error estimate of the proposed approximation is estimated in the next lemma.
If $a_0\in D((-\Delta)^{\gamma})$, $\gamma\in(\frac{1}{\alpha},1]$ and $a_1,f\in L^2(\Omega)$, based on the regularity result of Lemma \ref{thm-reg-L2-v}, the estimate of $\mathcal{R}_v^n$ can be deduced as follows.

\begin{lemma}\label{lem-r-1}
For $\mathcal{R}_v^n$ in (\ref{error-system}), it holds that
\[
\|\mathcal{R}_v^n\|_{L^2} \le Ct_n^{-\beta}N^{-q_1}\Big(\|(-\Delta)^{\gamma}a_0\|_{L^2}+\|a_1\|_{L^2}+\|f\|_{L^2}\Big),~~q_1=\min\{r(1-\beta),2-\beta\}.
\]
\end{lemma}
\begin{proof}
For $n=1$, it holds that
\begin{align*}
\|\mathcal{R}_v^1\|_{L^2}
&\le C\bigg\|\int_0^{\tau_1}(\tau_1-s)^{-\beta}\bigg(v'(s)-\frac{v^1-v^0}{\tau_1}\bigg)ds\bigg\|_{L^2}\\
&\le C\bigg\|\int_0^{\tau_1}(\tau_1-s)^{-\beta}\bigg(v'(s)-\frac{1}{\tau_1}\int_0^{\tau_1}v'(y)dy\bigg)ds\bigg\|_{L^2}\\
&\le C\tau_1^{-1}\bigg\|\int_0^{\tau_1}(\tau_1-s)^{-\beta}\int_0^{\tau_1}\big(v'(s)-v'(y)\big)dyds\bigg\|_{L^2}\\
&\le C\tau_1^{-1}\int_0^{\tau_1}(\tau_1-s)^{-\beta}\int_0^{\tau_1}\big(\|v'(s)\|_{L^2}+\|v'(y)\|_{L^2}\big)dyds\\
&\le C\tau_1^{-1}\int_0^{\tau_1}(\tau_1-s)^{-\beta}\int_0^{\tau_1}\big(s^{-\beta}+y^{-\beta}\big)dyds\\
&\le C\tau_1^{1-2\beta}.
\end{align*}
For the case $n\geq2$, one has
\begin{align*}
\|\mathcal{R}_v^n\|_{L^2}
&\le C\bigg\|\sum_{k=0}^{n-1}\int_{t_k}^{t_{k+1}}(t_n-s)^{-\beta}\bigg(v'(s)-\frac{v^{k+1}-v^k}{\tau_{k+1}}\bigg)ds\bigg\|_{L^2}\\
&\le C\sum_{k=0}^{n-1}\bigg\|\int_{t_k}^{t_{k+1}}(t_n-s)^{-\beta}\bigg(v'(s)-\frac{v^{k+1}-v^k}{\tau_{k+1}}\bigg)ds\bigg\|_{L^2}\\
&:=C\sum_{k=0}^{n-1}\mathcal{R}_v^{n,k}.
\end{align*}
Then, we consider the estimate of $\mathcal{R}_v^n$ term by term. For the case $k=0$, it gives that
\begin{align*}
\mathcal{R}_v^{n,0}
&\le C\bigg\|\int_{0}^{t_1}(t_n-s)^{-\beta}\bigg(v'(s)-\frac{v^1-v^0}{\tau_1}\bigg)ds\bigg\|_{L^2}\\
&\le C\int_{0}^{t_1}(t_n-s)^{-\beta}\|v'(s)\|_{L^2}ds+C\tau_1^{-1}\int_0^{t_1}(t_n-s)^{-\beta}\int_0^{t_1}\|v'(y)\|_{L^2}dyds\\
&\le C(t_n-t_1)^{-\beta}\int_0^{t_1}s^{-\beta}ds+C\tau_1^{-1}\int_0^{t_1}(t_n-s)^{-\beta}\int_0^{t_1}y^{-\beta}dyds\\
&\le C(t_n-t_1)^{-\beta}\tau_1^{1-\beta}+C\tau_1^{-\beta}\int_0^{t_1}(t_n-s)^{-\beta}ds\\
&\le C(t_n-t_1)^{-\beta}\tau_1^{1-\beta}\\
&\le Ct_n^{-\beta}\tau_1^{1-\beta},
\end{align*}
where $(t_n-t_1)^{-\beta}=t_n^{-\beta}(1-t_1/t_n)^{-\beta}\le Ct_n^{-\beta}$.

For the case $k= n-1$, following identity is used
\begin{align*}
v'(s)-\frac{v^{n}-v^{n-1}}{\tau_{n}}=\frac{1}{\tau_{n}}\int_{t_{n-1}}^{t_{n}}v'(s)-v'(y)dy=\frac{1}{\tau_{n}}\int_{t_{n-1}}^{t_{n}}\int_{y}^sv''(z)dzdy.
\end{align*}
From Lemma \ref{thm-reg-L2-v}, it holds that
\begin{align*}
\bigg\|v'(s)-\frac{v^{n}-v^{n-1}}{\tau_{n}}\bigg\|_{L^2}
\le \frac{1}{\tau_{n}}\int_{t_{n-1}}^{t_{n}}\int_{\min\{s,y\}}^{\max\{s,y\}}\|v''(z)\|_{L^2}dzdy\le C\tau_{n}\sup_{s\in(t_{n-1},t_n)} s^{-\beta-1}.
\end{align*}
And, one has
\begin{align*}
\mathcal{R}_v^{n,n-1}
&\le C\bigg\|\int_{t_{n-1}}^{t_{n}}(t_n-s)^{-\beta}\bigg(v'(s)-\frac{v^{n}-v^{n-1}}{\tau_{n}}\bigg)ds\bigg\|_{L^2}\\
&\le C\tau_{n}\sup_{s\in(t_{n-1},t_n)} s^{-\beta-1}\int_{t_{n-1}}^{t_{n}}(t_n-s)^{-\beta}ds\\
&\le C\tau_n^{2-\beta}\sup_{s\in(t_{n-1},t_n)} s^{-\beta-1}\\
&\le C\bigg(\frac{\tau_n}{t_n}\bigg)^{2-\beta}t_n^{1-2\beta}\\
&\le Ct_n^{-\beta}N^{-q_1}t_n^{-q_1/r}t_n^{1-\beta}\\
&\le Ct_n^{-\beta}N^{-q_1},
\end{align*}
the inequality holds using $\tau_n\leq CN^{-1}t_n^{1-\frac{1}{r}}$.
For the case $1\le k\le n-2$, we have
\begin{align*}
\mathcal{R}_v^{n,k}
&\le C\bigg\|\int_{t_{k}}^{t_{k+1}}(t_n-s)^{-\beta}\bigg(v'(s)-\frac{v^{k+1}-v^{k}}{\tau_{k+1}}\bigg)ds\bigg\|_{L^2}\\
&\le C\int_{t_{k}}^{t_{k+1}}(t_n-s)^{-\beta-1}\bigg\|v(s)-\frac{(s-t_k)v^{k+1}+(t_{k+1}-s)v^{k}}{\tau_{k+1}}\bigg\|_{L^2}ds\\
&\le C\tau_{k+1}^2\sup_{s\in(t_k,t_{k+1})}s^{-\beta-1}\int_{t_{k}}^{t_{k+1}}(t_n-s)^{-\beta-1}ds\\
&\le C\bigg(\tau_{k+1}^{2-\beta}t_{k+1}^\beta\sup_{s\in(t_k,t_{k+1})}s^{-\beta-1}\bigg)\bigg(\tau_{k+1}^\beta t_{k+1}^{-\beta}\int_{t_k}^{t_{k+1}}(t_n-s)^{-\beta-1}ds\bigg)\\
&\le C\bigg(\tau_{k+1}^{2-\beta}t_{k+1}^\beta\sup_{s\in(t_k,t_{k+1})}s^{-\beta-1}\bigg)\bigg(\tau_{k+1}^\beta \int_{t_k}^{t_{k+1}}s^{-\beta}(t_n-s)^{-\beta-1}ds\bigg).
\end{align*}
Denote $\Phi_k:=\tau_{k+1}^{2-\beta}t_{k+1}^\beta\sup_{s\in(t_k,t_{k+1})}s^{-\beta-1}\sim N^{-q_1}$. Hence, we obtain
\begin{align*}
\sum_{k=1}^{n-2}\mathcal{R}_v^{n,k}
&\le C\sum_{k=1}^{n-2}\Phi_k\tau_{k+1}^\beta \int_{t_k}^{t_{k+1}}s^{-\beta}(t_n-s)^{-\beta-1}ds\\
&\le C\max_k\Phi_k\bigg(\tau_{n-1}^\beta \int_{t_1}^{t_{n-1}}s^{-\beta}(t_n-s)^{-\beta-1}ds\bigg)\\
&\le Ct_n^{-\beta}N^{-q_1},
\end{align*}
the last inequality holds by
\begin{align*}
\int_{t_1}^{t_{n-1}}s^{-\beta}(t_n-s)^{-\beta-1}ds
&\le \int_{t_1}^{t_{n}/2}s^{-\beta}(t_n-s)^{-\beta-1}ds+\int_{t_{n}/2}^{t_{n-1}}s^{-\beta}(t_n-s)^{-\beta-1}ds\\
&\le C(t_n/2)^{1-\beta}(t_n-t_n/2)^{-\beta-1}+C(t_n/2)^{-\beta}\tau_n^{-\beta}\\
&\le Ct_n^{-\beta}\tau_n^{-\beta}+Ct_n^{-\beta}\tau_n^{-\beta}\\
&\le Ct_n^{-\beta}\tau_{n-1}^{-\beta}.
\end{align*}
The proof completes based on above estimates.
\end{proof}

Following the idea in Lemma \ref{lem-r-1}, if $a_0\in D((-\Delta)^{\gamma_1+\epsilon})$, $\gamma_1\in(\frac{1}{\alpha}+\frac14,\frac32]$ and $a_1,f\in D((-\Delta)^{\gamma_2+\epsilon})$, $\gamma_2\in(\frac{1}{4},1]$ and $0<\epsilon \ll 1$, from the regularity of $u$ in Lemma \ref{thm-reg-delta-2}, we get the estimate for $\mathcal{R}_u^n$ through the following lemma.
\begin{lemma}\label{lem-r-2}
For $\mathcal{R}_u^n$ in (\ref{error-system}), it holds that
\[
\|\nabla \mathcal{R}_u^n\|_{L^2} \le Ct_n^{-\beta}N^{-q_2}\Big(\|(-\Delta)^{\gamma_1+\epsilon}a_0\|_{L^2}
+\|(-\Delta)^{\gamma_2+\epsilon}a_1\|_{L^2}+\|(-\Delta)^{\gamma_2+\epsilon}f\|_{L^2}\Big),
\]
where $q_2=\min\{r(1-\beta/2),2-\beta\}.$
\end{lemma}

\textbf{For Alikhanov approximation $(\theta=\beta/2)$:}
The global consistency error estimate of the Alikhanov approximation is estimated in the
next lemmas.
\begin{lemma}{\rm{(\cite[Lemma 3.6 and Theorem 3.9]{LiaoH2021L2_1})}}\label{lem-r-v}
    Assume that \( v \in C^3((0, T]) \) and there exists a constant \( C>0 \) such that
    \[
    \|v'''(t)\|_{L^2} \leq C (1 + t^{(1-\beta)-3}), \quad 0 < t \leq T.
    \]
    Then
 \begin{align*}
    \sum_{j=1}^n P_{n-j}^{(n)} \|\mathcal{R}_v^{j-\theta}\|_{L^2} &\leq 
    C\left( \frac{\tau_1^{1-\beta}}{1-\beta} + t_1^{(1-\beta)-3} \tau_2^3 + 
    \frac{1}{1 - \beta} \max_{2 \leq k \leq n} \frac{t_k^\beta t_{k-1}^{(1-\beta)-3} \tau_k^3}{\tau_{k-1}^\beta} \right)\\
    &\leq CN^{-\min\{r(1-\beta),2\}}.
 \end{align*}

    Similarly, we yield
 \begin{align*}    
    \sum_{j=1}^n P_{n-j}^{(n)} \|\nabla\mathcal{R}_u^{j-\theta}\|_{L^2} &\leq  C\left( \frac{\tau_1^{1-\beta/2}}{1-\beta/2} + t_1^{(1-\beta/2)-3} \tau_2^3 +  \frac{1}{1-\beta} \max_{2 \leq k \leq n} \frac{t_k^\beta t_{k-1}^{(1-\beta/2)-3} \tau_k^3}{\tau_{k-1}^\beta} \right)\\
    &\leq CN^{-\min\{r(1-\beta/2),2\}}.     
 \end{align*}
\end{lemma}

If $a_0\in D((-\Delta)^{\gamma})$, $\gamma\in(\frac{1}{\alpha}+\frac12,\frac32]$ and $a_1,f\in H^1(\Omega)$, from the regularity of $\nabla v$ in Lemma \ref{thm-reg-delta-v} and \cite[Lemma 3.8 and Theorem 3.9]{LiaoH2021L2_1}, we have the following lemma on estimating
the time weighted approximation.
\begin{lemma}\label{lem-nabla-v}
    Denote the local truncation error of $v^{n-\theta}$ (here $\theta=\beta/2$) by
    \[
    R_v^{n-\theta}=v(t_{n-\theta})-v^{n-\theta}\quad\text{for }1\leq n\leq N.
    \]
    Then the error satisfies
    \[
    \|\nabla R_v^{n-\theta}\|_{L^2}\leq C\Bigl(\tau_{1}^{1-\beta}/(1-\beta)+\max_{2\leq k\leq n}t_{k-1}^{(1-\beta)-2}\tau_{k}^{2}\Bigr)\leq CN^{-\min\{r(1-\beta),2\}}.
    \]
    
    In a similar fashion, we have
    \[     \|\Delta R_u^{n-\theta}\|_{L^2}\leq C\Bigl(\tau_{1}^{(1-\beta/2)}/(1-\beta/2)+\max_{2\leq k\leq n}t_{k-1}^{(1-\beta/2)-2}\tau_{k}^{2}\Bigr)\leq CN^{-\min\{r(1-\beta/2),2\}}.     \]
\end{lemma}

Next, the convergence analysis 
 of scheme (\ref{num-scheme-2}) will be given. Denote $\bar u^n:=u^n-U^n$ and $\bar z^n:=z^n-Z^n$. One has the error system of scheme (\ref{num-scheme-2}): 
\begin{equation}\label{error-system-2} 
\begin{cases}   (\bar\partial_t^\beta \bar z)^{n-\theta} - \Delta \bar u^{n-\theta} = \mathcal{R}_z^{n-\theta}+\Delta R_u^{n-\theta}, & 1\le n\le N, \\   
\bar z^{n-\theta}  = (\bar\partial_t^\beta \bar u)^{n-\theta} + \mathcal{R}_u^{n-\theta}+R_z^{n-\theta}, & 1\le n\le N, \\   
\bar u(x,0) = \bar z(x,0)=0, & x \in \Omega, \\   
\bar u(x,t) = \bar z(x,t)=0, & (x,t) \in \partial \Omega \times (0,T). \end{cases} 
\end{equation}

The error estimate of $\mathcal{R}_z^{n-\theta}$ and $\nabla R_z^{n-\theta}$ can be derived by Lemmas \ref{lem-r-1}, \ref{lem-r-v}  and \ref{lem-nabla-v}.

 The regularity assumptions are proposed for the $L1$ and $L2$-$1_\sigma$ methods as follows.
\begin{itemize}
    \item[\textbf{B1.}] $a_0\in D((-\Delta)^{\gamma_1+\epsilon})$ and $a_1,f\in D((-\Delta)^{\gamma_2+\epsilon})$,  $0<\epsilon \ll 1$.
    \item[\textbf{B2.}] $a_0\in D((-\Delta)^{\gamma_1+\epsilon+\frac12})$ and $a_1,f\in D((-\Delta)^{\gamma_2+\epsilon+\frac12})$, 
    \item[\textbf{B3.}] $a_0\in D((-\Delta)^{\max\{\gamma_1+\epsilon,1\}})$ and $a_1,f\in D((-\Delta)^{\max\{\gamma_2+\epsilon,\frac12\}})$, 
    \item[\textbf{B4.}] $a_0\in D((-\Delta)^{\max\{\gamma_1+\epsilon,1\}+\frac12})$  and $a_1,f\in D((-\Delta)^{\max\{\gamma_2+\epsilon,\frac12\}+\frac12})$,
\end{itemize}
where $\gamma_1\in(\frac{1}{\alpha}+\frac{1}{4},\frac32]$, $\gamma_2\in(\frac{1}{4},1]$, $0<\epsilon \ll 1$.
\begin{theorem}\label{thm-conv-1}
If the assumptions \textbf{B1} and \textbf{B2} hold for $L1$ and $L2$-$1_\sigma$ respectively, the numerical scheme (\ref{error-system}) are unconditional convergent with
\[
\|\bar v^n\|_{L^2}+\|\nabla \bar u^n\|_{L^2} \le
\begin{cases}
C(N^{-\min\{r(1-\beta),2-\beta\}}+N^{-\min\{r(1-\beta/2),2-\beta\}}), \quad \text{if } \theta=0;\\
C(N^{-\min\{r(1-\beta),2\}}+N^{-\min\{r(1-\beta/2),2\}}),\quad \text{if } \theta=\beta/2;
\end{cases}
\]
for $1 \leq n \leq N$.
\end{theorem}
\begin{proof}
Acting $\nabla$ on the second equation of (\ref{error-system}).
Taking the inner product with $\bar v^n$ and $-\nabla \bar u^n$ for the first two equations, respectively. It gives that
\begin{align*}
&((\bar\partial_t^\beta \bar v)^{n-\theta},\bar v^{n-\theta}) - (\Delta \bar u^{n-\theta},\bar v^{n-\theta}) = (\mathcal{R}_v^{n-\theta}+\Delta R_u^{n-\theta},\bar v^{n-\theta}), \\
&-(\nabla\bar v^{n-\theta},\nabla \bar u^{n-\theta}) 
= -((\bar\partial_t^\beta \nabla\bar u)^{n-\theta},\nabla \bar u^{n-\theta})-(\nabla\mathcal{R}_u^{n-\theta}+\nabla R_v^{n-\theta},\nabla \bar u^{n-\theta}).
\end{align*}
Adding above equations, it comes that
\begin{align*}
((\bar\partial_t^\beta \bar v)^{n-\theta},\bar v^{n-\theta}) + ((\bar\partial_t^\beta \nabla \bar u)^{n-\theta},\nabla \bar u^{n-\theta}) = (\mathcal{R}_v^{n-\theta}+\Delta R_u^{n-\theta},\bar v^{n-\theta})-(\nabla \mathcal{R}_u^{n-\theta}+\nabla R_v^{n-\theta},\nabla \bar u^{n-\theta}).
\end{align*}
By Lemma \ref{FD-ineq}, one has
\begin{align*}
&\frac12\bar\partial_t^\beta \big(\|\bar v^n\|_{L^2}^2 + \|\nabla \bar u^n\|_{L^2}^2\big) \\
\le& \big(\|\mathcal{R}_v^{n-\theta}\|_{L^2}+\|\Delta R_u^{n-\theta}\|_{L^2}\big)\|\bar v^{n-\theta}\|_{L^2}+\big(\|\nabla \mathcal{R}_u^{n-\theta}\|_{L^2}+\|\nabla R_v^{n-\theta}\|_{L^2}\big)\|\nabla \bar u^{n-\theta}\|_{L^2}\\
\le&\big(\|\mathcal{R}_v^{n-\theta}\|_{L^2}+\|\Delta R_u^{n-\theta}\|_{L^2}+\|\nabla \mathcal{R}_u^{n-\theta}\|_{L^2}+\|\nabla R_v^{n-\theta}\|_{L^2}\big)\big(\|\bar v^{n-\theta}\|_{L^2}+\|\nabla \bar u^{n-\theta}\|_{L^2}\big).
\end{align*}
The desired result follows from Lemmas \ref{grownwall} and \ref{P}.
\end{proof}

For (\ref{error-system-2}), the convergence estimate is obtained by Theorem \ref{thm-conv-1} similarly.
\begin{theorem}\label{thm-conv-2}
If the assumptions \textbf{B3} and \textbf{B4} hold for $L1$ and $L2$-$1_\sigma$ respectively, the numerical scheme (\ref{error-system-2}) are unconditional convergent with 
\[ \|\bar z^n\|_{L^2}+\|\nabla \bar u^n\|_{L^2} \le 
\begin{cases} C(N^{-\min\{r\beta,2-\beta\}}+N^{-\min\{r(1-\beta/2),2-\beta\}}), \quad \text{if } \theta=0;\\ 
C(N^{-\min\{r\beta,2\}}+N^{-\min\{r(1-\beta/2),2\}}),\quad \text{if } \theta=\beta/2; 
\end{cases} \] for $1 \leq n \leq N$. 
\end{theorem}

\subsection{$L^2(\Omega)$ norm estimate for the case of discontinue $a_1(x)$}
In the previous context, we give the $H^1(\Omega)$ norm error estimate, see Theorems \ref{thm-conv-1} and \ref{thm-conv-2}. The generalized results of the $L^2(\Omega)$ norm are considered in this section.

\begin{lemma}\label{lem-L2-v}
For the stability of (\ref{eq-gov-trans}), one has    
$$\|v\|_{H^{-1}(\Omega)} + \|u\|_{L^{2}(\Omega)}\leq C(\|a_0\|_{L^{2}(\Omega)}+\|a_1\|_{H^{-1}(\Omega)}+\|f\|_{H^{-1}(\Omega)}),$$
where $C$ is a constant depends on $t$ and $\alpha$.
\end{lemma}
\begin{proof}
Acting $(-\Delta)^{-\frac12}$ on the first two equations of (\ref{eq-gov-trans}).
Taking the inner product $(\cdot,\cdot)_{L^2(\Omega)}$ with $(-\Delta)^{-\frac12}v$ and $-(-\Delta)^{\frac12} u$ for the first two equations, respectively. It gives that
\begin{align*}
(\partial_t^\beta (-\Delta)^{-\frac12}v,(-\Delta)^{-\frac12}v) + ((-\Delta)^{\frac12} u,(-\Delta)^{-\frac12}v) &= \omega_{2-\alpha}(t)((-\Delta)^{-\frac12}a_1,(-\Delta)^{-\frac12}v)\\&~~~~+\Gamma{(2-\alpha)}\omega_{2-\alpha}(t)((-\Delta)^{-\frac12}f,(-\Delta)^{-\frac12}v), \\
-((-\Delta)^{-\frac12}v,(-\Delta)^{\frac12} u)  &= -(\partial_t^\beta (-\Delta)^{-\frac12}u,(-\Delta)^{\frac12} u). 
\end{align*}
Adding above equations, from (\ref{frac-power-define}), it comes that
\begin{align*}
(\partial_t^\beta v,v)_{H^{-1}(\Omega)} + (\partial_t^\beta u, u)_{L^{2}(\Omega)}
&= \omega_{2-\alpha}(t)(a_1,v)_{H^{-1}(\Omega)}+\Gamma{(2-\alpha)}\omega_{2-\alpha}(t)(f,v)_{H^{-1}(\Omega)}\\
&\le \omega_{2-\alpha}(t)(\|a_1\|_{H^{-1}(\Omega)}+\Gamma{(2-\alpha)}\|f\|_{H^{-1}(\Omega)})(\|v\|_{H^{-1}(\Omega)}+\| u\|_{L^{2}(\Omega)}).
\end{align*}
By Lemma \ref{Coercivity}, one has
\begin{align*}
&\|v\|_{H^{-1}(\Omega)}^2 + \| u\|_{L^{2}(\Omega)}^2\\
\le& \|a_0\|_{L^{2}(\Omega)}^2+2(\|a_1\|_{H^{-1}(\Omega)}+\Gamma{(2-\alpha)}\|f\|_{H^{-1}(\Omega)})\mathcal{I}^{\beta}\omega_{2-\alpha}(t)(\|v\|_{H^{-1}(\Omega)}+\|u\|_{L^{2}(\Omega)}),
\end{align*}
where $\mathcal{I}^{\beta}\omega_{2-\alpha}(t)=\omega_{2-\alpha/2}(t)$.
The desired result follows by the triangle inequality.
\end{proof}

\begin{lemma}\label{lem-L2-z}
For the stability of (\ref{eq-gov-trans-novel}), one has    
$$\|z\|_{H^{-1}(\Omega)} + \|u\|_{L^2(\Omega)}\leq C(\|a_0\|_{L^2(\Omega)}+\|a_1\|_{L^2(\Omega)}+\|f\|_{L^2(\Omega)}),$$
where $C$ is a constant depends on $t$ and $\alpha$.
\end{lemma}
\begin{proof}
Taking the inner product $(\cdot,\cdot)_{L^2(\Omega)}$ with $(-\Delta)^{-1}z$ and $-u$ for the first two equations, respectively. It gives that
\begin{align*}
(\partial_t^\beta z,(-\Delta)^{-1}z) + ((-\Delta) u,(-\Delta)^{-1}z) &= 0, \\
-( z, u)  &= -(\partial_t^\beta  u, u)+\omega_{2-\alpha/2}(t)( a_1+\Gamma{(2-\alpha)} f, u). 
\end{align*}
Adding above equations, from (\ref{frac-power-define}), it comes that
\begin{align*}
(\partial_t^\beta z,z)_{H^{-1}(\Omega)} + (\partial_t^\beta  u, u)_{L^2(\Omega)}
&= \omega_{2-\alpha/2}(t)( a_1+\Gamma{(2-\alpha)} f, u)\\
&\le \omega_{2-\alpha/2}(t)(\| a_1\|_{L^2(\Omega)}+\Gamma{(2-\alpha)}\| f\|_{L^2(\Omega)})(\|z\|_{H^{-1}(\Omega)}+\| u\|_{L^2(\Omega)}).
\end{align*}
By Lemma \ref{Coercivity}, one has
\begin{align*}
&\|z\|_{H^{-1}(\Omega)}^2 + \|u\|_{L^2(\Omega)}^2\\
\le& \|a_0\|_{L^2(\Omega)}+2(\|a_1\|_{L^2(\Omega)}+\Gamma{(2-\alpha)}\|f\|_{L^2(\Omega)})\mathcal{I}^{\beta}\omega_{2-\alpha/2}(t)(\|z\|_{H^{-1}(\Omega)}+\|u\|_{L^2(\Omega)}),
\end{align*}
where $\mathcal{I}^{\beta}\omega_{2-\alpha/2}(t)=\omega_{2}(t)$.
The desired result follows by the triangle inequality.
\end{proof}

\begin{remark}
Comparing the results in Section \ref{sec-stability}, $L^2(\Omega)$ norm stability is derived based on the more general assumption for $a_1(x)\in H^{-1}(\Omega)$ or $L^2(\Omega)$.   
\end{remark}

\begin{remark} If $a_0(x)\in H^2(\Omega)\cap H_0^1(\Omega)$, $a_1(x)\in L^2(\Omega)$ and $t^{\alpha-1}g(x,t)\in L^2(\Omega)$ for $t\in(0,T]$, then the error systems (\ref{error-system}) and (\ref{error-system-2}) (here $\theta=0$) are unconditional convergent with 
\[  \begin{cases} 
 \|\bar v^n\|_{H^{-1}(\Omega)}+\| \bar u^n\|_{L^2(\Omega)} \le C(N^{-\min\{r(1-\beta),2-\beta\}}+N^{-\min\{r(1-\beta/2),2-\beta\}}), \quad \text{for (\ref{error-system})};\\  
 \|\bar z^n\|_{H^{-1}(\Omega)}+\| \bar u^n\|_{L^2(\Omega)} \le C(N^{-\min\{r\beta,2-\beta\}}+N^{-\min\{r(1-\beta/2),2-\beta\}}),\quad \text{for (\ref{error-system-2})};  
 \end{cases} \] 
 where $1 \leq n \leq N$.
 In order to obtain the desired results, we act $(-\Delta)^{-\frac12}$ the first two equations on the error equation (\ref{error-system}), and then take the inner product $(\cdot,\cdot)_{L^2(\Omega)}$ with $(-\Delta)^{-\frac12}\bar v^n$ and $-(-\Delta)^{\frac12} \bar u^n$ for the first two equations, respectively. For error system (\ref{error-system-2}) , we take the inner product $(\cdot,\cdot)_{L^2(\Omega)}$ with $(-\Delta)^{-1}\bar z^n$ and $-\bar u^n$ for the first two equations, respectively.
A discrete fractional Gr\"{o}nwall inequality proposed in Lemma \ref{grownwall} is a crucial tool
in the numerical analysis of the given problems.   
\end{remark}

\section{Numerical experiment}\label{sec-num}
In this section, we carry out some numerical experiments using finite element method to check the theoretical results. Let us revisit the foundational work in which the SFOR method was originally proposed \cite{LyuP2022SFOR}. As established in Remark 2.2 of that paper, auxiliary variables were  precisely introduced to avoid explicitly representing the singular temporal kernel $a_1(x)\omega_{2-\alpha}(t)$. Based on this, they adopted the following numerical framework with non-singular source $g(x,t)$:
\begin{equation}\label{eq-gov-ref}
\begin{cases}
  \partial_t^\beta \textbf{v}  - \Delta \textbf{u} = t\Delta a_1(x)+g(x,t), & (x,t) \in \Omega \times (0,T), \\
  \textbf{v}  = \partial_t^\beta \textbf{u}, & (x,t) \in \Omega \times (0,T), \\
  \textbf{u}(x,0) = \textbf{v}(x,0)=0, & x \in \Omega, \\
  \textbf{u}(x,t) = \textbf{v}(x,t)=0, & (x,t) \in \partial \Omega \times (0,T),
\end{cases}
\end{equation}
where $u=\textbf{u}+ta_1(x)$. The limitation of (\ref{eq-gov-ref}) is $(\Delta a_1(x),\phi(x))$, $x\in\Omega$ must exist, where $\phi(x)$ is a basis function from a finite element space.
Our numerical frameworks (\ref{eq-gov-trans}) and (\ref{eq-gov-trans-novel}) can relax this requirement. First, we find that the optimal convergence of $L1$ scheme is reached, despite the presence of $a_1(x)\omega_{2-\alpha}(t)$, when the mesh parameter of graded meshes is $r=\frac{4-\alpha}{2-\alpha}$ for $1<\alpha<2-\epsilon$, where $\epsilon$ is a fixed positive constant.

Theoretically, the scheme becomes inapplicable as $\alpha\rightarrow2^-$ because the mesh parameter $r=\frac{4-\alpha}{2-\alpha}$ diverges to $\infty$, rendering the underlying functional framework ill-defined. When $\Delta a_1\in L^2(\Omega)$, the reference scheme (\ref{eq-gov-ref}) remains well-posed and effective. For more general $a_1(x)$, the transformed framework (\ref{eq-gov-trans}) is recommended, as it imposes a bounded $r$ and thereby ensures stability and applicability across a broader class of problems. Furthermore, we notice that $a_1(x)\omega_{2-\alpha}(t)$ and singular source $g(x,t)$ $\rightarrow$ $\infty$ as $t\rightarrow0^+$. Therefore, a new form (\ref{eq-gov-trans-novel}) is considered, which is well-posed and validate when $\nabla a_1(x)\in L^2(\Omega)$, see Section \ref{sec-stability}. Collectively, these three different numerical frameworks provide a systematic and theoretically grounded protocol for deploying the SFOR method under varying regularity assumptions.

The $L1$ and $L2$-$1_\sigma$ methods are employed to simulate the problems (\ref{eq-gov-trans}) and (\ref{eq-gov-trans-novel}). To rigorously validate the theoretical properties of the proposed schemes, we conduct two carefully designed numerical experiments. We examine the temporal convergence behavior of both methods under varying fractional order $\alpha \in (1,2)$ and mesh grading parameter $r\geq 1$, with results summarized in Tables \ref{table-test1}-\ref{table7}. First, Tables \ref{table-test1} and \ref{table-test2} confirm that the $L1$ scheme applied to system (\ref{eq-gov-trans}) achieves the predicted convergence rate of $2-\beta$ when $r = \frac{4-\alpha}{2-\alpha}$—a value derived from our error analysis to compensate for solution singularity near $t=0$. In contrast, under uniform temporal discretization ($r=1$), the observed convergence order drops below one, corroborating the necessity of graded meshes for optimal accuracy. Second, we apply both the $L1$ and $L2$-$1_\sigma$ methods to system (\ref{eq-gov-trans-novel}) using their respective theoretically optimal mesh parameters with different parameter $\alpha$; the resulting errors, reported in Tables \ref{table-test3} and \ref{table-test4}, align precisely with the sharp convergence rates established in our error bounds. Finally, a second numerical experiment, presented in Tables \ref{table5}-\ref{table7}, further validates the robustness of our theoretical convergence results under varying fractional orders $\alpha$, mesh grading exponents $r$ and solution regularity assumptions.

\begin{example}\label{ex-test}
Problems (\ref{eq-gov}) with $\Omega=(0,\pi)$, $T=1$, 
$a_0(x)=a_1(x)=x$ for $x\in(0,\pi/2],$ $a_0(x)=a_1(x)=\pi-x$ for $x\in(\pi/2,\pi)$, and $f(x)=\sin(x)$.
The size of the space grids $h=\frac{\pi}{100}$, $N$ is the number of partitions in the time
grids. $e_{H^1}=\max_{1\leq n\leq N}\|u_h^n-U_h^n\|_{H^1}$, where $u_h^n$ and $U_h^n$ are the reference solution ($h=\frac{\pi} {100}$ and $N=2560$, $128$ for $L1$, $L2$-$1_\sigma$, respectively) and the numerical solution, respectively. Furthermore, to test the convergence rate, let $Order=\log_2(e_{H^1}(N/2)/e_{H^1}(N))$.
\end{example}

\begin{table}[!ht]
\caption{$L1$ scheme for Example \ref{ex-test} applied to system (\ref{eq-gov-trans}) with $\alpha=1.25$.}\label{table-test1}
\renewcommand{\arraystretch}{1}
\def\temptablewidth{0.95\textwidth}
\begin{center}
 \begin{tabular*}{\temptablewidth}{@{\extracolsep{\fill}}lcccccc}\hline
 $N$ &\multicolumn{2}{c}{$r=1$} &\multicolumn{2}{c}{$r_{opt}=\frac{4-\alpha}{2-\alpha}$} &\multicolumn{2}{c}{$r=\frac{4-\alpha}{\alpha}$}    \\
 \cline{2-3}\cline{4-5}\cline{6-7}
 & $e_{H^1}(N)$         &Order     &$e_{H^1}(N)$    &Order    &$e_{H^1}(N)$ &Order     \\ \hline
 20 &1.2199e-01     &-       &4.4765e-02   &-        &4.3459e-02  &-  \\
 40 &9.3580e-02     &0.3825  &1.7882e-02   &1.3238   &2.7938e-02  &0.6374    \\
 80 &7.4264e-02     &0.3335  &6.9907e-03   &1.3550   &1.4523e-02  &0.9439 \\
 160 &5.9328e-02    &0.3240  &2.6843e-03   &1.3809   &6.5413e-03  &1.1507    \\
 320 &4.7945e-02    &0.3073  &1.0031e-03   &1.4202   &2.6884e-03  &1.2828    \\\hline
 Optimal Order  &\multicolumn{6}{c}{1.375} \\ \hline
\end{tabular*}
\end{center}
\end{table}

\begin{table}[!ht]
\caption{$L1$ scheme for Example \ref{ex-test} applied to system (\ref{eq-gov-trans}) with $\alpha=1.75$.}\label{table-test2}
\renewcommand{\arraystretch}{1}
\def\temptablewidth{0.95\textwidth}
\begin{center}
 \begin{tabular*}{\temptablewidth}{@{\extracolsep{\fill}}lcccccc}\hline
 $N$ &\multicolumn{2}{c}{$r=1$} &\multicolumn{2}{c}{$r_{opt}=\frac{4-\alpha}{2-\alpha}$} &\multicolumn{2}{c}{$r=2$}    \\
 \cline{2-3}\cline{4-5}\cline{6-7}
 & $e_{H^1}(N)$         &Order     &$e_{H^1}(N)$    &Order    &$e_{H^1}(N)$ &Order     \\ \hline
 20 &1.5567e-00     &-       &1.1021e-00   &-        &1.3206e-00  &-  \\
 40 &1.2120e-00     &0.3611  &5.7352e-01   &0.9423   &9.0835e-01  &0.5399    \\
 80 &9.1647e-01     &0.4032  &2.8278e-01   &1.0201   &6.0692e-01  &0.5817 \\
 160 &6.6585e-01    &0.4609  &1.3683e-01   &1.0473   &3.9122e-01  &0.6335    \\
 320 &4.5430e-01    &0.5516  &6.4692e-02   &1.0808   &2.3828e-01  &0.7153    \\\hline
 Optimal Order  &\multicolumn{6}{c}{1.125} \\ \hline
\end{tabular*}
\end{center}
\end{table}

\begin{table}[!ht]
\caption{$L1$ scheme for Example \ref{ex-test} applied to system (\ref{eq-gov-trans-novel}) with $r_{opt} =\max\{\frac{4-\alpha}{\alpha},2\}$.}\label{table-test3}
\renewcommand{\arraystretch}{1}
\def\temptablewidth{0.95\textwidth}
\begin{center}
 \begin{tabular*}{\temptablewidth}{@{\extracolsep{\fill}}lcccccc}\hline
 $N$  &\multicolumn{2}{c}{$\alpha=1.25$} &\multicolumn{2}{c}{$\alpha=1.5$} &\multicolumn{2}{c}{$\alpha=1.75$}\\
 \cline{2-3}\cline{4-5}\cline{6-7}
 & $e_{H^1}(N)$         &Order     &$e_{H^1}(N)$    &Order    &$e_{H^1}(N)$ &Order     \\ \hline   
 20  &4.3482e-02   &-       &4.5973e-02   &-        &1.4843e-01  &-  \\
 40  &2.7953e-02   &0.6374  &3.1962e-02   &0.5244   &7.5944e-02  &0.9668    \\
 80  &1.4531e-02   &0.9439  &1.7218e-02   &0.8925   &3.7619e-02  &1.0135 \\
 160 &6.5447e-03   &1.1507  &8.2308e-03   &1.0648   &1.7910e-02  &1.0707    \\
 320 &2.6897e-03   &1.2829  &3.6316e-03   &1.1804   &8.0457e-03  &1.1545  \\\hline
 Theoretical Order
 &\multicolumn{2}{c}{1.375} &\multicolumn{2}{c}{1.25}&\multicolumn{2}{c}{1.125}\\ \hline
\end{tabular*}
\end{center}
\end{table}

\begin{table}[!ht]
\caption{$L2$-$1_\sigma$ scheme for Example \ref{ex-test} applied to system (\ref{eq-gov-trans-novel}) with $r_{opt} =\max\{\frac{4}{\alpha},\frac{8}{4-\alpha}\}$.}\label{table-test4}
\renewcommand{\arraystretch}{1}
\def\temptablewidth{0.95\textwidth}
\begin{center}
 \begin{tabular*}{\temptablewidth}{@{\extracolsep{\fill}}lcccccc}\hline
 $N$  &\multicolumn{2}{c}{$\alpha=1.25$} &\multicolumn{2}{c}{$\alpha=1.5$} &\multicolumn{2}{c}{$\alpha=1.75$}\\
 \cline{2-3}\cline{4-5}\cline{6-7}
 & $e_{H^1}(N)$         &Order     &$e_{H^1}(N)$    &Order    &$e_{H^1}(N)$ &Order     \\ \hline   
 8    &4.9369e-02   &-       &5.9829e-02   &-        &1.4207e-01  &-  \\
 16   &2.3531e-02   &1.0690  &3.0216e-02   &0.9856   &5.9822e-02  &1.2479    \\
 32  &7.3752e-03   &1.6738  &9.0900e-03   &1.7329   &2.3059e-02  &1.3754 \\
 64  &1.5651e-03   &2.2364  &1.8637e-03   &2.2861   &5.7428e-03  &2.0055  \\\hline
 Theoretical Order
 &\multicolumn{2}{c}{2} &\multicolumn{2}{c}{2}&\multicolumn{2}{c}{2}\\ \hline
\end{tabular*}
\end{center}
\end{table}

\begin{example}\label{ex2}
Problems (\ref{eq-gov}) with $\Omega=(0,\pi)$, $T=1$,
$a_0(x)=x$ for $x\in(0,\pi/2]$ and $a_0(x)=\pi-x$ for $x\in(\pi/2,\pi),$ $a_1(x)=\chi_{(0,\frac{\pi}{2}]}(x),$ $f(x)=\sin(x)$. $e_{L^2}=\max_{1\leq n\leq N}\|u_h^n-U_h^n\|_{L^2}$.
We use the numerical solution with the size of the space grids $h=\frac{\pi}{100}$ and $N=2560$ being the number of partitions in the time grids as the reference solution for $L1$.
\end{example}

\begin{table}[!ht]
\caption{$L1$ scheme for Example \ref{ex2} applied to system (\ref{eq-gov-trans}) with $\alpha=1.25$.}\label{table5}
\renewcommand{\arraystretch}{1}
\def\temptablewidth{0.95\textwidth}
\begin{center}
 \begin{tabular*}{\temptablewidth}{@{\extracolsep{\fill}}lcccccc}\hline
 $N$ &\multicolumn{2}{c}{$r=1$} &\multicolumn{2}{c}{$r_{opt}=\frac{4-\alpha}{2-\alpha}$} &\multicolumn{2}{c}{$r=\frac{4-\alpha}{\alpha}$}    \\
 \cline{2-3}\cline{4-5}\cline{6-7}
 & $e_{L^{2}}(N)$         &Order     &$e_{L^{2}}(N)$    &Order    &$e_{L^{2}}(N)$ &Order     \\ \hline
 20 &8.8265e-02     &-       &2.8315e-02   &-        &2.9771e-02  &-  \\
 40 &5.4054e-02     &0.7075  &1.1272e-02   &1.3288   &1.2518e-02  &1.2499    \\
 80 &3.2242e-02     &0.7454  &4.3937e-03   &1.3593   &5.1100e-03  &1.2926 \\
 160 &1.8617e-02    &0.7923  &1.6844e-03   &1.3832   &2.0334e-03  &1.3294    \\
 320 &1.0204e-02    &0.8675  &6.2885e-04   &1.4215   &7.8166e-04  &1.3793    \\\hline
 Optimal Order  &\multicolumn{6}{c}{1.375} \\ \hline
\end{tabular*}
\end{center}
\end{table}

\begin{table}[!ht]
\caption{$L1$ scheme for Example \ref{ex2} applied to system (\ref{eq-gov-trans}) with $\alpha=1.75$.}\label{table6}
\renewcommand{\arraystretch}{1}
\def\temptablewidth{0.95\textwidth}
\begin{center}
 \begin{tabular*}{\temptablewidth}{@{\extracolsep{\fill}}lcccccc}\hline
 $N$ &\multicolumn{2}{c}{$r=1$} &\multicolumn{2}{c}{$r_{opt}=\frac{4-\alpha}{2-\alpha}$} &\multicolumn{2}{c}{$r=2$}    \\
 \cline{2-3}\cline{4-5}\cline{6-7}
 & $e_{L^{2}}(N)$         &Order     &$e_{L^{2}}(N)$    &Order    &$e_{L^{2}}(N)$ &Order     \\ \hline
 20 &1.3571e-00     &-       &9.4906e-01   &-        &1.1505e-00  &-  \\
 40 &1.0571e-00     &0.3604  &4.8916e-01   &0.9562   &7.9178e-01  &0.5391    \\
 80 &7.9962e-00     &0.4027  &2.3762e-01   &1.0417   &5.2929e-01  &0.5810 \\
 160 &5.8106e-01    &0.4606  &1.1328e-01   &1.0687   &3.4132e-01  &0.6330    \\
 320 &3.9649e-01    &0.5514  &5.3117e-02   &1.0927   &2.0794e-01  &0.7149    \\\hline
 Optimal Order  &\multicolumn{6}{c}{1.125} \\ \hline
\end{tabular*}
\end{center}
\end{table}

\begin{table}[!ht]
\caption{$L1$ scheme for Example \ref{ex2} applied to system (\ref{eq-gov-trans-novel}) with $r_{opt} =\max\{\frac{4-\alpha}{\alpha},2\}$.}\label{table7}
\renewcommand{\arraystretch}{1}
\def\temptablewidth{0.95\textwidth}
\begin{center}
 \begin{tabular*}{\temptablewidth}{@{\extracolsep{\fill}}lcccccc}\hline
 $N$  &\multicolumn{2}{c}{$\alpha=1.25$} &\multicolumn{2}{c}{$\alpha=1.5$} &\multicolumn{2}{c}{$\alpha=1.75$}\\
 \cline{2-3}\cline{4-5}\cline{6-7}
 & $e_{L^{2}}(N)$         &Order     &$e_{L^{2}}(N)$    &Order    &$e_{L^{2}}(N)$ &Order     \\ \hline   
 20  &7.2407e-03   &-       &2.0167e-02   &-        &1.0373e-01  &-  \\
 40  &2.9740e-03   &1.2837  &8.8484e-03   &1.1885   &4.8803e-02  &1.0879    \\
 80  &1.1799e-03   &1.3337  &3.7861e-03   &1.2247   &2.2536e-02  &1.1147 \\
 160 &4.5664e-04   &1.3696  &1.5830e-03   &1.2581   &1.0185e-02  &1.1457    \\
 320 &1.7136e-04   &1.4140  &6.4006e-04   &1.3064   &4.4395e-03  &1.1980  \\\hline
 Theoretical Order
 &\multicolumn{2}{c}{1.375} &\multicolumn{2}{c}{1.25}&\multicolumn{2}{c}{1.125}\\ \hline
\end{tabular*}
\end{center}
\end{table}

\newpage
\section{Conclusions}
In this work, we propose a symmetric fractional-order reduction (SFOR) method for solving fractional wave equations with low regularity on nonuniform temporal meshes. By coupling classical $L$-type discretizations—including the $L1$ and $L2$-$1_\sigma$ methods—with newly derived optimal parameter choices specifically designed for nonuniform grids, we construct unconditionally convergent and high-order numerical algorithms. A rigorous error analysis demonstrates that the SFOR method attains optimal convergence rates under significantly relaxed regularity assumptions on the exact solution. Comprehensive numerical experiments corroborate both the theoretical convergence orders and the method's robustness across diverse mesh grading strategies. This framework thus provides a theoretically grounded and computationally reliable approach for simulating challenging fractional wave dynamics. Future work will explore extensions to multidimensional spatial domains and broader families of time-fractional and space-time-fractional PDEs.

\section*{Declarations}
On behalf of all authors, the corresponding author states that there is no conflict of interest. No datasets were generated or analyzed during the current study.

\bibliographystyle{plain}
\bibliography{ref}

\end{document}